\documentclass[11pt]{article}
\usepackage[a4paper,margin=1in]{geometry}
\usepackage{amssymb}
\usepackage{amsmath}
\usepackage{amsfonts}
\usepackage{amsthm}
\usepackage[all]{xy}

\newcommand{\C}{\mathbb{C}}
\newcommand{\Z}{\mathbb{Z}}

\newcommand{\N}{\mathbb{N}}

\newcommand{\G}{\mathcal{G}}
\newcommand{\E}{\mathcal{E}}
\newcommand{\Odd}{\mathcal{O}}

\newtheorem{theorem}{Theorem}

\newtheorem{conjecture}{Conjecture}
\newtheorem{corollary}{Corollary}

\newtheorem{definition}{Definition}
\newtheorem{example}{Example}

\newtheorem{lemma}{Lemma}

\newtheorem{proposition}{Proposition}

\numberwithin{equation}{section}

\begin{document}
\title{\textbf{Schur Ring, Run Structure and Periodic Compatible Binary Sequences}}
\author{Ronald Orozco L\'opez}

\newcommand{\Addresses}{{
  \bigskip
  \footnotesize

  \textsc{Department of Mathematics, Universidad de los Andes,
    Bogot\'a Colombia,}\par\nopagebreak
  \textit{E-mail address}, \texttt{rj.orozco@uniandes.edu.co}

}}

\maketitle

\begin{abstract}
In this paper three Schur ring are discussed, namenly: Hamming, circulant orbists and decimated
circulant orbits Schur ring. By using autocorrelation function and the run structure of binary
sequences we proof the relation between this Schur ring and combinatorial structures such as
Hadamard matrices, periodic compatible binary sequences and perfect binary sequences. Cai proved 
for binary sequences that the autocorrelation function is in fact completely determined by its 
run structure. Also, he characterised the structure of the circulant Hadamard matrices. We
characterise a more general structure, called periodic compatible binary sequences ($PComS$ for
brevety), which generalises Hadamard matrices, periodic complementary binary 
sequences and binary sequences with $2$-level autocorrelation. Families of periodic compatibles
binary sequences are presented. Also, we compute a bounds on familias $PComS$ in Hamming Schur ring. 
The results obtained are applied to families of $PComS$ such as circulant, with one and two 
circulant cores, Goethals-Seidel type and partial Hadamard matrices and perfect binary sequences.
\end{abstract}

\section{Introduction}

The concept of Schur ring ($S$-ring) was iniciated by I.Schur in their classical paper [1] which 
was published in 1933. Later, the theory of $S$-ring was developed for Wielandt [2].
But the main objective of theory was purely group theoretical concept, especially in
problem concerning the permutations groups. In the 80s and 90s, the theory received a notable 
impulse by the study of $S$-ring over cyclic groups and their applications to the graph theory 
[3],[4],[5],[6]. Later, Leung and Man obtained a complete clasification of 
$S$-ring over finite cyclic groups in four types, namely: trivial, orbits, dot products and 
wedge products [7],[8]. 

In this paper the Schur ring over $\Z_{2}^{n}$ are not classified, but three types are discussed,
namely: Hamming, Circulant orbists and Decimated circulant orbits Schur ring. By using
autocorrelation function and the run structure of binary
sequences we proof the relation between this Schur ring and combinatorial structures such as
Hadamard matrices, periodic compatible binary sequences and perfect binary sequences. Cai[11] proved 
for binary sequences that the autocorrelation function is in fact completely determined by its 
run structure. Also, he characterised the structure of the circulant Hadamard matrices. We
characterise a more general structure, called periodic compatible binary sequences ($PComS$ for
brevety), which generalises Hadamard matrices, periodic complementary binary 
sequences and binary sequences with $2$-level autocorrelation.

This paper is organized as follows. In Section 2, three types of Schur ring over $\Z_{2}^{n}$ 
are established: Hamming, Circulant orbists and Decimated circulant orbits Schur ring. Section 
3 relates composition of integer and circulant orbit Schur ring. In Section 4, characterization
of the autocorrelation-run structure for perfect binary sequences is presented. In Section 5,
families of periodic compatibles binary sequences are discussed, particularly families in 
$\Z_{2}^{n}$, $n=4,5,6,7,8,9$. Also, the relation with the $S$-sets of Hamming Schur ring is
stated. In Section 6, we compute a bounds on families $PComS$ in Hamming Schur ring. In Section 7, 
the results obtained are applied to families of $PComS$ such as circulant, with one and two 
circulant cores, Goethals-Seidel type and partial Hadamard matrices and perfect binary sequences.

\section{Some Schur rings in $\mathbb{Z}_{2}^{n}$}

Let $G$ be a finite group with identity element $e$ and $\C[G]$ the group algebra of all formal sums 
$\sum_{g\in G}a_{g}g$, $a_{g}\in \C$, $g\in G$. For $T\subset G$, the element $\sum_{g\in T}g$ will 
be denoted by $\overline{T}$. Such an element is also called a $\textit{simple quantity}$. 
The transpose of $\alpha = \sum_{g\in G}a_{g}g$ is defined as $\alpha^{\top} = \sum_{g\in G}a_{g}(g^{-1})$. Let $\{T_{0},T_{1},...,T_{r}\}$ be a partition of $G$ and let $S$ be the subspace of $\C[G]$ spanned by $\overline{T_{1}},\overline{T_{2}},...,\overline{T_{r}}$.  We say that $S$ is a $\textit{Schur ring}$ ($S$-ring, for short) over $G$ if: 

\begin{enumerate}
\item $T_{0} = \lbrace e\rbrace$, 
\item for each $i$, there is a $j$ such that $T_{i}^{\top} = T_{j}$,
\item for each $i$ and $j$, we have $\overline{T_{i}}\overline{T_{j}} = \sum_{k=1}^{r}\lambda_{i,j,k}\overline{T_{k}}$, for constants $\lambda_{i,j,k}\in\C$.
\end{enumerate}

The numbers $\lambda_{i,j,k}$ are the structure constants of $S$ with respect to the linear base 
$\{\overline{T_{0}},\overline{T_{1}},...,\overline{T_{r}}\}$. 
The sets $T_{i}$ are called the \textit{basic sets} of the $S$-ring $S$. Any union of them is called an $S$-set. Thus, $X\subseteq G$ is an $S$-set if and only if $\overline{X}\in S$. The set of all $S$-set is closed with respect to taking inverse and product. Any subgroup of $G$ that is an $S$-set, is called an $S$-\textit{subgroup} of $G$ or $S$-\textit{group}.

In this paper denote by $\Z_{2}$ the cyclic group of order 2 with elements $+$
and $-$(where + and $-$ mean 1 and $-1$ respectively). Let 
$\Z_{2}^{n}=\overset{n}{\overbrace{\Z_{2}\times \cdots \times \Z_{2}}}$. Then all $X\in\Z_{2}^{n}$ are sequences of $+$ and $-$ and will be called $\Z_{2}-$\textit{sequences}. 

In this section, three types of Schur ring over $\Z_{2}^{n}$ are established: Hamming, circulant
orbists and decimated circulant orbits Schur ring.

\subsection{Hamming Schur ring}
Let $\omega(X)$ denote the Hamming weight of $X\in\Z_{2}^{n}$. Thus, $\omega(X)$ is the number of $+$ in any $\Z_{2}-$sequences $X$ of $\Z_{2}^{n}$. Now let $\G_{n}(k)$ be the subset of $\Z_{2}^{n}$ such that $\omega(X)=k$ for all $X\in\G_{n}(k)$, where $0\leq k\leq n$. 

We let $T_{i}=\G_{n}(n-i)$. It is straightforward to prove that the partition $S=\{\G_{n}(0),...,\G_{n}(n)\}$ is a partition of $\Z_{2}^{n}$. And also $S$ is an $S-$ring over 
$\Z_{2}^{n}$. From [9] it is know that
\begingroup\makeatletter\def\f@size{12}\check@mathfonts
\begin{equation}\label{producto}
\G_{n}(a)\G_{n}(b)=
\begin{cases}
\bigcup\limits_{i=0}^{a}\G_{n}(n-a-b+2i), & 0\leq a\leq\left\lfloor\dfrac{n}{2}\right\rfloor, a\leq b\leq n-a,\\
\bigcup\limits_{i=0}^{n-a}\G_{n}(a+b-n+2i), & \left\lfloor\dfrac{n}{2}\right\rfloor+1\leq a\leq n, n-a\leq b\leq a.  
\end{cases}
\end{equation}
\endgroup

The above $S$-ring is known as Hamming Schur ring and denoted by $S_{H}$. This Schur ring
is $(n+1)$-dimensional. The union of all sets $\G_{n}(2a)$ in $S_{H}$ will be called 
\textbf{the even partition} of $S_{H}$, and will be designated by $\E_{n}$. The 
\textbf{odd partition} $\Odd_{n}$ is defined analogously. The sets $\E_{2n}$ and $\Odd_{2n+1}$ are 
$S_{H}$-subgroups of $S_{H}$ of order $2^{2n-1}$ and $2^{2n}$, respectively.

\subsection{Circulant Orbit Schur ring}

Let $\left\langle C\right\rangle\leq Aut(\mathbb{Z}_{2}^{n})$ denote the cyclic permutation group of
order $n$, that is, $C(x_{i})=x_{(i+1) mod n}$. Let $X_{C}=Orb_{\left\langle C\right\rangle }X=\{C^{i}(X):C^{i}\in \left\langle C\right\rangle \}$. Therefore, $\left\langle C\right\rangle$
defines a partition in equivalent class on $\Z_{2}^{n}$ and this we will denote by $\Z_{2C}^{n}$.
The partition of $\Z_{2}^{n}$ given by $\left\langle C\right\rangle$ defines a Schur ring $\Z_{2C}^{n}$, denoted $S_{C}$, where each $X_{C}$ in $\Z_{2C}^{n}$ will be called \textbf{circulant basic set}. An $S$-set de $\Z_{2C}^{n}$ will be called circulant $S$-set. 

In general, $\left\vert X_{C}\right\vert\neq n$, $X\neq\pm(1,1,...,1)$. If 
$\left\vert X_{C}\right\vert=n$, we say that $\left\langle C\right\rangle$ acts freely on 
$X\in \Z_{2}^{n}$ and let $F(\Z_{2C}^{n})$ denote the $S$-set of all this. Similarly, denote 
$\widehat{F}(\Z_{2C}^{n})$ the $S$-set all $X_{C}$ such that $\left\langle C\right\rangle$ 
doesn't act freely on $X$. Let $\pm(1,1,...,1)$ be in both $F(\mathbb{Z}_{2C}^{n})$ and 
$\widehat{F}(\mathbb{Z}_{2C}^{n})$. Therefore, 
\begin{equation}\label{Z_libre}
\Z_{2C}^{n}=F(\Z_{2C}^{n})\cup \widehat{F}(\Z_{2C}^{n}).
\end{equation}
When $n=p$ is an odd prime number $\left\vert 
\G_{p}(a)\right\vert$ is divisible by $p$, $0\leq a\leq p$, therefore 
$\left\vert X_{C}\right\vert =p$ for all $X_{C}\in \Z_{2C}^{p}$ and $\widehat{F}(\Z_{2C}^{p})=\emptyset$. Let $\widehat{F}_{d}(\Z_{2C}^{n})$ denote the set 
\begin{equation}
\{X_{n/d,C}=(\overbrace{X,X,\dots,X}^{n/d})_{C}\mid X\in \Z_{2}^{d}\}, 
\end{equation}
where $d\mid n$. Then
\begin{equation}
\widehat{F}(\Z_{2C}^{n})=\bigcup_{d\mid n}\widehat{F}_{d}(\Z_{2C}^{n})
\end{equation}
and 
\begin{equation}
\widehat{F}_{d_{1}}(\Z_{2C}^{n})\widehat{F}_{d_{2}}(\Z_{2C}^{n})=\widehat{F}_{d_{1}d_{2}}(\Z_{2C}^{n}),
\end{equation}
where $d_{1},d_{2}\mid n$ and $d_{1}d_{n}<n$. If $d_{1}\mid d_{2}$, then
\begin{equation}
\widehat{F}_{d_{1}}(\Z_{2C}^{n})\widehat{F}_{d_{2}}(\Z_{2C}^{n})=\widehat{F}_{d_{2}}(\Z_{2C}^{n}).
\end{equation}
If $d_{1}=d_{2}$, then
\begin{equation}
\widehat{F}_{d_{1}}(\Z_{2C}^{n})\widehat{F}_{d_{1}}(\Z_{2C}^{n})=\widehat{F}_{d_{1}}(\Z_{2C}^{n}).
\end{equation}
If $d_{1}d_{2}=n$, 
\begin{equation}
\widehat{F}_{d_{1}}(\Z_{2C}^{n})\widehat{F}_{d_{2}}(\Z_{2C}^{n})=F(\Z_{2C}^{n}).
\end{equation}
Finally, if $d\mid n$, then
\begin{equation}
\widehat{F}_{d}(\Z_{2C}^{n})F(\Z_{2C}^{n})=F(\Z_{2C}^{n}).
\end{equation}

Then $\widehat{F}_{d}(\Z_{2C}^{n})$ is an $S_{C}$-subgroup of $S_{C}$ for all $d\vert n$. 
On the other hand, in [9] was shown that $X_{C}^{2}\not\subset F(\mathbb{Z}_{2C}^{n})$ 
only if $n$ is an even number. Also, $X_{C}Y_{C}\subset F(\mathbb{Z}_{2C}^{n})$ if $n$ is an
odd number. Thus, $F(\mathbb{Z}_{2C}^{n})$ is an $S_{C}$-subgroup in $S_{C}$ only if $n$ is an 
odd number.\\

Next, we calculate the dimension of Schur ring $S_{C}$ over $\Z_{2}^{n}$ when $n=p$ and $n=p^{k}$,
$p$ an odd prime
\begin{theorem}
If $S_{C}$ is a circulant orbit Schur ring in $\mathbb{Z}_{2}^{p}$, $p$ an odd prime, then
$S_{C}$ has dimension
\begin{equation}\label{dimension_circular_p}
\vert S_{C}\vert=\frac{2^{p}+2p-2}{p}
\end{equation}
\end{theorem}
\begin{proof}
We know that $p\mid\binom{p}{a}$ for $1\leq a\leq p-1$. Thereby,
\begin{eqnarray*}
\vert S_{C}\vert&=&1+\frac{1}{p}\sum_{a=1}^{p-1}\binom{p}{a}+1=2+\frac{2^{p}-2}{p}
\end{eqnarray*}
and from here we will get the right result.
\end{proof}

\begin{theorem}
If $S_{C}$ is a circulant orbit Schur ring in $\mathbb{Z}_{2}^{p^{n}}$, $p$ an odd prime, then
$S_{C}$ has dimension
\begin{eqnarray}\label{dimension_circular_power_p}
\vert S_{C}\vert&=&2+\frac{1}{p^{n}}\sum_{\overset{(a,p^{n})=1}{a>1}}\binom{p^{n}}{a} \nonumber
+\sum_{i=1}^{n}\frac{1}{p^{n-i+1}}\sum_{\overset{a_{i}=1}{p\nmid a_{i}}}^{p^{n-i}-1}\left[\binom{p^{n-i+1}}{a_{i}p}-\binom{p^{n-i}}{a_{i}}\right]\nonumber\\
&&+\sum_{i=1}^{n}\frac{1}{p^{n-i+1}}\sum_{k=2}^{n-i}\sum_{a_{i}=1}^{p-1}\left[\binom{p^{n-i+1}}{a_{i}p^{k-i+1}}-\sum_{j=1}^{k}\binom{p^{n-j-i}}{a_{i}p^{k-j-i}}\right]
\end{eqnarray}
\end{theorem}
\begin{proof}
First, we know that $p^{n}\mid\binom{p^{n}}{a}$ only if $(a,p^{n})=1$. Followed, we define
$\G_{n}^{m}(a)=\{(X,X,...,X):X\in\G_{n}(a)\}$. If $1\leq l\leq n$ and $p\nmid a$, then is clear 
that $\G_{p^{l}}(ap)\supset\G_{p^{l-1}}^{p}(a)$ and $p^{l}\mid\left[\binom{p^{l}}{ap}-\binom{p^{l-1}}{a}\right]$. Now, let $k,l$ be integer numbers such that $1\leq k<l\leq n$ and we make $s=ap^{k}$, $1\leq a\leq p-1$. Then $\G_{p^{l}}(s)=\G_{p^{l}}(ap^{k})\supset\G_{p^{l-1}}^{p}(ap^{k-1})\supset\G_{p^{l-2}}^{p^{2}}(ap^{k-2})\supset\cdots\supset\G_{p^{l-k}}^{p^{k}}(a)$. Then $p^{l}\mid\left[\binom{p^{l}}{ap^{k}}-\sum_{j=1}^{k}\binom{p^{l-j-1}}{ap^{k-j}}\right]$. Account taken of 
the above, the result is obtained.
\end{proof}

On the other hand, pick $X\in\mathbb{Z}_{2}^{n}$. Denote $RX$ the reversed sequence 
$RX = (x_{n-1},...,x_{1},x_{0})$. We will call $X_{C}$ symmetric if exists $Y\in X_{C}$ such that
$RY=Y$. In otherwise, we say it is nonsymmetric. Denote $Sym(\Z_{2C}^{n})$ the set of all $X_{C}$
symmetric and $\widehat{Sym}(\Z_{2C}^{n})$ the set of all $X_{C}$ nonsymmetric. Then the Schur 
ring $\Z_{2C}^{n}$ it can to express as 
\begin{equation}\label{Z_sym}
Z_{2C}^{n}=Sym(Z_{2C}^{n})\cup \widehat{Sym}(Z_{2C}^{n}).
\end{equation}
From (\ref{Z_libre}) and (\ref{Z_sym}) it follow that
\begin{eqnarray}
F(\Z_{2C}^{n})&=&Sym(F(\Z_{2C}^{n}))\cup \widehat{Sym}(F(\Z_{2C}^{n}))\\
\widehat{F}(\Z_{2C}^{n})&=&Sym(\widehat{F}(\Z_{2C}^{n}))\cup \widehat{Sym}(\widehat{F}(\Z_{2C}^{n}))\\
Sym(\Z_{2C}^{n})&=&F(Sym(\Z_{2C}^{n}))\cup \widehat{F}(Sym(\Z_{2C}^{n}))\\
\widehat{Sym}(\Z_{2C}^{n})&=&F(\widehat{Sym}(\Z_{2C}^{n}))\cup \widehat{F}(\widehat{Sym}(\Z_{2C}^{n})).
\end{eqnarray}
Hence
\begin{eqnarray}
\Z_{2C}^{n}&=&Sym(F(\Z_{2C}^{n}))\cup\widehat{Sym}(F(\Z_{2C}^{n}))\cup Sym(\widehat{F}(\Z_{2C}^{n}))\cup\widehat{Sym}(\widehat{F}(\Z_{2C}^{n})).
\end{eqnarray}

Now, we show that if $X_{C}\in Sym(F(\mathbb{Z}_{2C}^{n}))$, then $X_{C}^{2}\subset Sym(F(\mathbb{Z}_{2C}^{n}))$. By definition, there exists $Y$ in $X_{C}$ such that $RY=Y$. Suppose without loss of generality that $Y=X$. Also, is clear that $F(\mathbb{Z}_{2C}^{n})$ is an $S_{C}$-subgroup only if
$n$ is odd number. Then, only we need to prove that $(XC^{i}X)_{C}$ is symmetric for all $i$. But
it's easy to verify that $C^{(n-1)/2}(XCX)$ is symmetric. From here, it followed that 
$C^{k(n-1)/2}(XC^{k}X)$ is symmetric for all $k$. Therefore $X_{C}^{2}\subset Sym(F(\mathbb{Z}_{2C}^{n}))$. However, in general $Sym(F(\mathbb{Z}_{2C}^{n}))$ is not $S_{C}$-subgroup in $S_{C}$. 
For example, if $X=---+---$ and $Y=+-+++-+$, then $XC^{3}Y$ is not symmetric. Equally, in general $Sym(\widehat{F}_{d}(\mathbb{Z}_{2C}^{n}))$ is not an $S_{C}$-subgroup because if both $X$ and $Y$ have the preceding definition, then $(X,X,X)C^{3}((Y,Y,Y))$ is not symmetric.

\subsection{Decimated Circulant Orbit Schur ring}

Let $\delta_{k}\in S_{n-1}$ act on $X\in\Z_{2}^{n}$ by decimation, that is, $\delta_{k}(x_{i})=x_{ki}$ for all $x_{i}$ in $X$, $(k,n)=1$ and let $\Delta_{n}$ denote the set of this $\delta_{k}$. The set $\Delta_{n}$ is a group of order $\phi(n)$ isomorphic to $U(\Z_{n})$, the group the units of $\Z_{n}$, where $\phi$ is called the Euler totient function. We define 
\begin{equation}\label{decimated_schur_ring}
\Delta_{n}(\mathbb{Z}_{2C}^{n})=\{\Delta_{n}(X_{C}):X_{C}\in\mathbb{Z}_{2C}^{n}\}
\end{equation}
where
\begin{equation}
\Delta_{n}(X_{C})=\{(\delta_{r}X)_{C}:\delta_{r}\in\Delta_{n}\}
\end{equation}
The results about of the action of $\Delta_{n}$ on $\mathbb{Z}_{2C}^{n}$ are listed below
\begin{enumerate}
\item If $X_{C}\in\widehat{F}(\mathbb{Z}_{2C}^{n})$, then $\Delta_{n}(X_{C})\in\widehat{F}(\mathbb{Z}_{2C}^{n})$.
\item If $X_{C}\in F(\mathbb{Z}_{2C}^{n})$, then $\Delta_{n}(X_{C})\in F(\mathbb{Z}_{2C}^{n})$.
\item If $X_{C}\in Sym(\mathbb{Z}_{2C}^{n})$, then $\Delta_{n}(X_{C})\in Sym(\mathbb{Z}_{2C}^{n})$.
\item If $X_{C}\in\widehat{Sym}(\mathbb{Z}_{2C}^{n})$, then $\Delta_{n}(X_{C})\in\widehat{Sym}(\mathbb{Z}_{2C}^{n})$.
\end{enumerate}
Thus $\Delta_{n}$ define a partition on of Schur ring $\mathbb{Z}_{2C}^{n}$ that preserves both the
freeness and the symmetry. It is easy to show that $(\ref{decimated_schur_ring})$ define a Schur ring which we will call \textit{decimated circulant orbit Schur ring}

\begin{theorem}
Let $S_{D}$ denote $\Delta_{n}(\mathbb{Z}_{2C}^{n})$. Then $S_{D}$ is a Schur ring over 
$\mathbb{Z}_{2}^{n}$ of dimension
\begin{equation}\label{dimension_decimated}
\vert S_{D}\vert=\frac{1}{n\phi(n)}\sum_{i=0}^{n-1}\sum_{\overset{k=1}{(n,k)=1}}^{n-1}2^{C(k,t)}
\end{equation}
\end{theorem}
\begin{proof}
It is clear that $\Z_{2C}^{n}$ ia a Schur ring. The remainder part of the proof follows by bearing 
in mind that $\left\langle C\right\rangle\Delta_{n}$ is a group. Note also that
$\Delta_{n}(X_{C})=\{(x_{ri+j})_{i=0}^{n-1}:0\leq j\leq n-1,\ (r,n)=1\}$, that is to say, 
$\Delta_{n}(X_{C})$ contains equivalent sequences under the action of $\left\langle C\right\rangle\Delta_{n}$. Hence the dimension of $S_{D}$ corresponds with the number of classes
$\Delta_{n}(X_{C})$. For the proof of (\ref{dimension_decimated}) see [12].
\end{proof}

\section{Correspondence between Composition of Integers and Circulant Orbit Schur Ring}

Let $\mathbb{N}$ be the set of all positive integers. A composition of $n\in\N$ is any sequence
$c_{1}c_{2}\cdots c_{m}$ of positive integers such that $\sum_{i=1}^{m}c_{i}=n$. The $c_{i}$ are called the parts and $m$ denote the number of parts or lenght of the composition. In this paper we
will denote with $\mathrm{C}(n)$ the set of composition of $n$ and $\mathrm{C}_{r}(n)$ the set of composition of $n$ with lenght $r$. We wish to establish a correspondence between compositions 
of integers and ciculant orbit Schur ring. Let $X$ be a $\Z_{2}$-sequence. The \textit{run vector} 
of $X$ is defined as a string of consecutives run of $+$ and $-$ with total lenght $k$, denoted 
by $R_{k}$. For instance, if $X=+++--+-+-$, then the run vector of $X$ is $R_{3}R_{2}R_{1}R_{1}R_{1}R_{1}=(3,2,1,1,1,1)$. Also, $X_{C}$ contains the following run vector 
\begin{eqnarray*}
&&(3,2,1,1,1,1),(2,2,1,1,1,1,1),(1,2,1,1,1,1,2),(2,1,1,1,1,3),\\
&&(1,1,1,1,1,3,1),(1,1,1,1,3,2),(1,1,1,3,2,1),(1,1,3,2,1,1),\\
&&(1,3,2,1,1,1).
\end{eqnarray*}
We define the length of a circulant $X_{C}$, denote $l(X_{C})$, as the minimum lenght of the 
run vector in $X_{C}$. It is clear that $l(X_{C})$ is an even number so we are interested in the
following definitions

\begin{definition}
Let $X=(x_{1},...,x_{r})\in\mathrm{C}_{r}(n)$ and let $Y=(y_{1},...,y_{r})\in\mathrm{C}_{r}(m)$. We define the \textbf{r-alternating product} $\diamond_{r}$ of $X$ and $Y$ as
\begin{eqnarray*}
X\diamond_{r}Y&=&(x_{1},...,x_{r})\diamond_{r}(y_{1},...,y_{r})\\
&=&(x_{1},y_{1},...,x_{r},y_{r})
\end{eqnarray*}
The collection of all $X\diamond_{r}Y$ is denoted by $\mathrm{C}_{r}(n)\diamond_{r}\mathrm{C}_{r}(m)$.
\end{definition}

\begin{definition}
Pick $X$ in $\G_{n}(a)$. We will say that $X$ is in \textbf{partitioned form} if either $X=(p_{1},m_{1},...,p_{r},m_{r})\in\mathrm{C}_{r}(a)\diamond_{r}\mathrm{C}_{r}(n-a)$ or $X=(m_{1},p_{1},...,m_{r},p_{r})\in \mathrm{C}_{r}(n-a)\diamond_{r}\mathrm{C}_{r}(a)$, with $p_{1}+\cdots+p_{r}=a$ and $m_{1}+\cdots+m_{r}=n-a$.
\end{definition}

Now, when $X$ is in partitioned form we will denoted their length as $l(X)$. Let $X$ be as defined above. Then $l(X)=6$ and $X_{C}$ contains the run vectors $(3,1,1)\diamond_{3}(2,1,1),
(2,1,1)\diamond_{3}(1,1,3),(1,1,3)\diamond_{3}(1,1,2),(1,1,2)\diamond_{3}(1,3,1),(1,2,1)\diamond_{3}(3,1,1)$. On the other hand, let $\left\langle c\right\rangle$ be a group acting cyclically over 
$\mathrm{C}_{r}(n)$, thus, $\left\langle c\right\rangle\leq Aut(\mathrm{C}_{r}(n))$ and let 
$\left\langle\overline{C}\right\rangle\leq Aut(\mathrm{C}_{r}(n)\diamond_{r}\mathrm{C}_{r}(m))$.
Then $\left\langle c\right\rangle$ and $\left\langle\overline{C}\right\rangle$ are linked as follows

\begin{proposition}
Let $P\in\mathrm{C}_{r}(n)$ and $Q\in\mathrm{C}_{r}(m)$. Then
\begin{enumerate}
\item $\overline{C}^{2i}(P\diamond_{r}Q)=c^{i}P\diamond_{r}c^{i}Q$
\item $\overline{C}^{2i+1}(P\diamond_{r}Q)=c^{i}Q\diamond_{r}c^{i+1}P$.
\end{enumerate}
\end{proposition}
\begin{proof}
Pick $P=(p_{1},p_{2},...,p_{r})\in\mathrm{C}_{r}(n)$ and $Q=(q_{1},q_{2},...,q_{r})\in\mathrm{C}_{r}(m)$. Then $P\diamond_{r}Q=(p_{1},q_{1},p_{2},q_{2},...,p_{r},q_{r})$ and
\begin{eqnarray*}
\overline{C}^{2i}(P\diamond_{r}Q)&=&(p_{i},q_{i},...,p_{r},q_{r},...,p_{i-1},q_{i-1})\\
&=&(p_{i},...,p_{r},...,p_{i-1})\diamond_{r}(q_{i},...,q_{r},...,q_{i-1})\\
&=&c^{i}P\diamond_{r}c^{i}Q
\end{eqnarray*}
and
\begin{eqnarray*}
\overline{C}^{2i+1}(P\diamond_{r}Q)&=&(q_{i},...,p_{r},q_{r},...,p_{i-1},q_{i-1},p_{i})\\
&=&(q_{i},...,q_{r},...,q_{i-1})\diamond_{r}(p_{i+1},...,p_{r},...,p_{i})\\
&=&c^{i}Q\diamond_{r}c^{i+1}P.
\end{eqnarray*}
as we hoped.
\end{proof}

The group $\left\langle\overline{C}\right\rangle$ define a partition over 
$\mathrm{C}_{r}(n)\diamond_{r}\mathrm{C}_{r}(m)$ and let $X_{\overline{C}}$ denote each 
equivalence class. If either $X\in\mathrm{C}_{r}(a)\diamond_{r}\mathrm{C}_{r}(n-a)$ or $X\in\mathrm{C}_{r}(n-a)\diamond_{r}\mathrm{C}_{r}(a)$, then $X_{\overline{C}}$ corresponds to $X_{C}$.
Hence the circulant orbit Schur ring $S_{C}$ in $\Z_{2}^{n}$ is characterized by the compositions of
$n$.

\section{Autocorrelation and Run String of Binary Sequences}

This section shows the relationship between the run string of a binary sequences and their
autocorrelacion. In this way, the reason for wanting to characterize circulant orbit Schur ring 
$S_{C}$ in $\Z_{2}^{n}$ is clear. For a periodical $\mathbb{Z}_{2}$-sequence $X=(x_{0},x_{1},...,x_{n-1})$ with period $n$, the \textit{periodic autocorrelation} at shift $k$ is the number defined by

\begin{equation*}
\mathsf{P}_{X}(k)=\sum_{i=0}^{n-1}x_{i}x_{i+k}.
\end{equation*}
It is well known that
\begin{eqnarray*}
\mathsf{P}_{X}(k)&=&\mathsf{P}_{X}(n-k),\\
\mathsf{P}_{X}(k)&\equiv& n \mod4,
\end{eqnarray*}
for $0\leq k\leq n-1$. It is obvious that $\mathsf{P}_{X}(0)=n$ and it is called 
\textit{trivial} autocorrelation value. $\mathsf{P}_{X}(k)$, $1\leq k\leq n-1$ are 
called \textit{nontrivial} autocorrelation values. Also, $\mathsf{P}_{C^{i}X}(k)=\mathsf{P}_{X}(k)$
for $1\leq i\leq n-1$. So we deal with class $X_{C}$ in the Schur ring $\mathbb{Z}_{2C}^{n}$ 
instead of binary sequences. With $\mathsf{P}_{X}(k)$ we refer to autocorrelation value of $X_{C}$.

On the other hand, let $\mathcal{R}$ be a collection of run string types and let $\mathcal{N}_{X}(\mathcal{R})$ denote the number of the runs string of $X$ which have a particular type in 
$\mathcal{R}$. In [11] can be found the following result which relates the autocorrelation of a
binary sequence with their run structure

\begin{theorem}
For any binary sequence $X$ with period $n$ and $k=1,2,3,...,n$,
\begin{equation}\label{equation_run}
\mathsf{P}_{X}(k)=n-2k\cdot l(X)-4\sum_{i_{1}+\cdots+i_{r}<k}(-1)^{r}(k-i)\mathcal{N}_{X}(R_{i_{1}}\cdots R_{i_{r}})
\end{equation}
where $i=i_{1}+i_{2}+\cdots+i_{r}$ and $l(X)$ is the lenght of run of $X$.
\end{theorem}
For $k=1,2,3,4,5$ we obtain the following autocorrelation values $\mathsf{P}_{X}(k)$
\begin{eqnarray*}
\mathsf{P}_{X}(1)&=&n-2l(X),\\
\mathsf{P}_{X}(2)&=&n-4l(X)+4\mathcal{N}_{X}(R_{1}),\\
\mathsf{P}_{X}(3)&=&n-6l(X)+8\mathcal{N}_{X}(R_{1})+4\mathcal{N}_{X}(R_{2})-4\mathcal{N}_{X}(R_{1}R_{1}),\\
\mathsf{P}_{X}(4)&=&n-8l(X)+12\mathcal{N}_{X}(R_{1})+8\mathcal{N}_{X}(R_{2})+4\mathcal{N}_{X}(R_{3})-8\mathcal{N}_{X}(R_{1}R_{1})\\
&&-4\mathcal{N}_{X}(R_{1}R_{2})-4\mathcal{N}_{X}(R_{2}R_{1})+4\mathcal{N}_{X}(R_{1}R_{1}R_{1}),\\
\mathsf{P}_{X}(5)&=&n-10l(X)+16\mathcal{N}_{X}(R_{1})+12\mathcal{N}_{X}(R_{2})+8\mathcal{N}_{X}(R_{3})+4\mathcal{N}_{X}(R_{4})\\
&&-12\mathcal{N}_{X}(R_{1}R_{1})-8\mathcal{N}_{X}(R_{1}R_{2})-8\mathcal{N}_{X}(R_{2}R_{1})+8\mathcal{N}_{X}(R_{1}R_{1}R_{1})\\
&&-4\mathcal{N}_{X}(R_{1}R_{3})-4\mathcal{N}_{X}(R_{3}R_{1})-4\mathcal{N}_{X}(R_{2}R_{2})+4\mathcal{N}_{X}(R_{1}R_{1}R_{2})\\
&&+4\mathcal{N}_{X}(R_{1}R_{2}R_{1})+4\mathcal{N}_{X}(R_{2}R_{1}R_{1})-4\mathcal{N}_{X}(R_{1}R_{1}R_{1}R_{1})
\end{eqnarray*}

It is easy to notice a pattern in the above values, so the equation (\ref{equation_run}) can 
be expressed as follows:
\begin{eqnarray}\label{run-new}
\mathsf{P}_{X}(k)&=&n-2kl(X)+4(k-1)\mathcal{N}_{X}(R_{1})\\
&&+\sum_{i=2}^{k-1}4(k-i)\left(\mathcal{N}_{X}(R_{i})-\sum_{\overset{i_{1}+\cdots+i_{r}=i}{1<r\leq i}}(-1)^{r}\mathcal{N}_{X}(R_{i_{1}}\cdots R_{i_{r}})\right)\nonumber
\end{eqnarray}
for all $k\geq3$. The proof it is easy. For $k\geq3$ we have by (\ref{equation_run})
\begin{eqnarray*}
\mathsf{P}_{X}(k)&=&n-2kl(X)-4\sum_{i_{1}+\cdots+i_{r}<k}(-1)^{r}(k-i)\mathcal{N}_{X}(R_{i_{1}}\cdots R_{i_{r}})\\
&=&n-2kl(X)-4\left(\sum_{i=1}^{k-1}\sum_{i_{1}+\cdots+i_{r}=i}(-1)^{r}(k-i)\mathcal{N}_{X}(R_{i_{1}}\cdots R_{i_{r}})\right)\\
&=&n-2kl(X)+4(k-1)\mathcal{N}_{X}(R_{1})+\\
&&4\left(\sum_{i=2}^{k-1}(k-i)\sum_{i_{1}+\cdots+i_{r}=i}(-1)^{r+1}\mathcal{N}_{X}(R_{i_{1}}\cdots R_{i_{r}})\right)\\
&=&n-2kl(X)+4(k-1)\mathcal{N}_{X}(R_{1})+\\
&&4\sum_{i=2}^{k-1}(k-i)\left(\mathcal{N}_{X}(R_{i})-\sum_{\overset{i_{1}+\cdots+i_{r}=i}{1<r\leq i}}(-1)^{r}\mathcal{N}_{X}(R_{i_{1}}\cdots R_{i_{r}})\right)
\end{eqnarray*}
as we wanted.

We will say that $X_{C}$ has \textit{2-level autocorrelation values} if all nontrivial
autocorrelation values are equal to some constant $d$. In the theorem 5.6 of [11] a 
characterization on Hadamard circulant matrices was obtained. We generalize this theorem 
for sequence with $2$-level autocorrelation values

\begin{theorem}\label{caracterizacion}
If $X_{C}$ has 2-level autocorrelation values with nontrivial autocorrelation values $d$, then
\begin{eqnarray}\label{numero_run_k}
l(X)&=&\frac{n-d}{2},\nonumber\\
\mathcal{N}_{X}(R_{1})&=&\frac{n-d}{4},\nonumber\\
\mathcal{N}_{X}(R_{k})&=&\sum_{\overset{i_{1}+\cdots+i_{r}=k}{1<r\leq k}}(-1)^{r}\mathcal{N}_{X}(R_{i_{1}}\cdots R_{i_{r}}),
\end{eqnarray}
with $n\equiv d\mod4$ and $2\leq k\leq n-2$.
\end{theorem}
\begin{proof}
$l(X)$ is easily obtained by replacing $\mathsf{P}_{X}(1)$ with $d$ in (\ref{run-new}), for $k=1$.
Making $k=2$ we obtain $\mathcal{N}_{X}(R_{1})$ by replacing $\mathsf{P}_{X}(1)$ with $d$ and by using the above result for $l(X)$. Equally, $\mathcal{N}_{X}(R_{2})$ is obtained by making 
$\mathsf{P}_{X}(3)=d$ and from the values of $l(X)$ and $\mathcal{N}_{X}(R_{1})$. Now, suppose 
we have calculated the first values $k$ of $\mathsf{P}_{X}$. Hence (\ref{numero_run_k}) is true
for $k-1$. From (\ref{run-new}) we have
\begin{eqnarray*}
d&=&\mathsf{P}_{X}(k+1)\\
&=&n-2(k+1)\left(\frac{n-d}{2}\right)+4k\left(\frac{n-d}{4}\right)\\
&&+\sum_{i=3}^{k}4(k+1-i)\left(\mathcal{N}_{X}(R_{i})-\sum_{\overset{i_{1}+\cdots+i_{r}=i}{1<r\leq i}}(-1)^{r}\mathcal{N}_{X}(R_{i_{1}}\cdots R_{i_{r}})\right)\\
&=&d+4\mathcal{N}_{X}(R_{k})-4\sum_{\overset{i_{1}+\cdots+i_{r}=k}{1<r\leq k}}(-1)^{r}\mathcal{N}_{X}(R_{i_{1}}\cdots R_{i_{r}})
\end{eqnarray*}
Therefore, (\ref{numero_run_k}) is also true for $k$. 
\end{proof}

The above result characterizes all class $X_{C}$ with $2$-level autocorrelation. In the following
section we will get any further ahead and we will to characterize families of binary sequences whose
sum of nontrivial autocorrelations has constant value $c$.

\section{Periodic Compatible Binary Sequences}

In [17] was defined two sequences $A$ and $B$ as compatible if the sum of their periodic nontrivial
autocorrelations is a constant, say $c$. That is
\begin{equation*}
\mathsf{P}_{A}(k)+\mathsf{P}_{B}(k)=c.
\end{equation*}
In this section we extend the above definition to any family of sequences with this property.

\begin{definition}
A family of circulant $S$-sets $\{A_{iC}\}$, $1\leq i\leq q$, in $\mathbb{Z}_{2C}^{n}$, that
satisfy $\sum_{i=1}^{q}\mathsf{P}_{A_{i}}(k)=c$ for $k\neq0$, it a family of \textbf{periodic compatible sequences} and denoted $PComS(n,q,c)$.
\end{definition}

When $c=0$ this family is called \textit{periodic complementary sequences}(see [13],[14]) 
and denoted $PCS(n,q)$. From this, $PComS(n,q,0)=PCS(n,q)$. We will prove some preliminary results 
on $PComS$.
\begin{theorem}
If there exist families $PComS(n,q_{1},c_{1})$ and $PComS(n,q_{2},c_{2})$, then there exists a
family $PComS(n,q_{1}+q_{1},c_{1}+c_{2})$. 
\end{theorem}

\begin{corollary}
If there exist $PComS(n,q,c)$, then there exists $PComS(n,kq,kc)$ for all $k\geq2$.
\end{corollary}

Then, we show a generalization of theorem \ref{caracterizacion} for a family $PComS(n,q,c)$ 

\begin{theorem}\label{caract_PComS}
Let $\{A_{iC}\}$ be a $PComS(n,q,c)$. Then
\begin{eqnarray*}
\sum_{i=1}^{q}l(A_{i})&=&\frac{nq-c}{2}\\
\sum_{i=1}^{q}\mathcal{N}_{A_{i}}(R_{1})&=&\frac{nq-c}{4}\\
\sum_{i=1}^{q}\mathcal{N}_{A_{i}}(R_{k})&=&\sum_{i=1}^{q}\sum_{\overset{i_{1}+\cdots+i_{r}=k}{1<r\leq k}}(-1)^{r}\mathcal{N}_{A_{i}}(R_{i_{1}}\cdots R_{i_{r}})
\end{eqnarray*}
for $2\leq k\leq n-2$.
\end{theorem}

Since $\G_{n}(a)=-\G_{n}(n-a)$ and $\mathsf{P}_{-X}(k)=\mathsf{P}_{X}(k)$, then will be enough
to find families $PComS$ in $\cup_{1<a\leq\lfloor\frac{n}{2}\rfloor}\G_{n}(a)$.

\begin{definition}
Let $\mathrm{I}$ be a any subset in $[0,1,...,n]$ and let $T(\mathrm{I})=\{\G_{n}(k)\}_{k\in\mathrm{I}}$ be a $S_{H}$-set of $\mathbb{Z}_{2}^{n}$. Pick a $\G_{n}(a)$ any in $S_{H}$. We say 
that $T(\mathrm{I})$ is $\G_{n}(a)$-\textbf{complete} if
\begin{enumerate}
\item $\G_{n}(i)\G_{n}(j)\supset\G_{n}(a)$ for all $\G_{n}(i),\G_{n}(j)\in T$.
\item There is no $\G_{n}(b)$ in $S_{H}$ such that $\G_{n}(b)^{2}\supset\G_{n}(a)$ and 
$\G_{n}(b)\G_{n}(k)\supset\G_{n}(a)$ for all $\G_{n}(k)\in T$.
\end{enumerate}
\end{definition}

\begin{theorem}
There is exactly a set $\G_{n}(a)$-complete $S_{H}$-set $T(\mathrm{I})$ of order $a$ with
$\mathrm{I}=[\frac{n-a}{2},\frac{n+a}{2}]$, $a<n$.
\end{theorem}
\begin{proof}
The proof is analogous to the proof of theorem 4 in [3].
\end{proof}

\begin{corollary}
If there exist $PComS(n,q,c)$, then is contained in a set $\G_{nq}(\frac{nq+c}{2})$-complete 
$S_{H}$-set.
\end{corollary}
\begin{proof}
Let $X=(A_{1},A_{2},...,A_{q})\in PComS(n,q,c)$ and suppose $X\in\G_{nq}(a)$. Then should be
$\G_{nq}(a)^{2}\supset\G_{nq}(\frac{nq+c}{2})$. From above theorem $a\in[\frac{nq-c}{4},\frac{3nq+c}{4}]$.
\end{proof}

\begin{theorem}\label{no_trivial_tamanio_1}
$PComS(n,1,n-4)=\G_{n}(1)$ for all $n$.
\end{theorem}
\begin{proof}
Pick $X$ in $PComS(n,1,n-4)$. Then $l(X)=2$ and $\mathcal{N}_{X}(R_{1})=1$. Hence 
$X=(1)\diamond(n-1)$ and $X\in\G_{n}(1)$. Now, pick $X$ in $\G_{n}(1)$. Then $\mathsf{P}_{X}(0)=n$
and $\mathsf{P}_{X}(k)=n-4$ for $1\leq k\leq n-1$. Thus $X\in PComS(n,1,n-4)$.
\end{proof}

Let $\mathfrak{A}(\Z_{2C}^{n})$ denote the set of \textit{autocorrelations vector} in $\Z^{n}$
and let $\theta:\Z_{2C}^{n}\rightarrow\mathfrak{A}(\Z_{2C}^{n})$ be the mapping 
$\theta(X_{C})=(\mathsf{P}_{X}(0),\mathsf{P}_{X}(1),\dots ,\mathsf{P}_{X}(n-1))$. The decimation group $\Delta_{n}$ do not alter the set of values which $\mathsf{P}_{X}(k)$ takes on, but merely the order in which they appear. This is, $XC^{i}X\longrightarrow XC^{ri}X$ and $\delta_{r}\in\Delta_{n}$ is a permutation over $\theta(X_{C})$. Let $\Delta_{n}(\theta(X_{C}))$ denote the set
\begin{equation*}
\{\delta_{r}(\theta (X_{C})):\delta_{r}\in \Delta_{n} \}\subset\mathfrak{A}(\Z_{2C}^{n}).
\end{equation*}
Then $\theta$ is a mapping of equivalence class, thus $\theta: \Delta_{n}(X_{C})\rightarrow \Delta_{n}(\theta(X_{C}))$. We have the commutative diagram
\begin{equation}\label{diagram}
\xymatrix{
 \Z_{2C}^{n} \ar[d]^{\theta} \ar[r]^{\delta_{r}} & \Z_{2C}^{n} \ar[d]^{\theta}\\
   \mathfrak{A}(\Z_{2C}^{n}) \ar[r]^{\delta_{r}} & \mathfrak{A}(\Z_{2C}^{n}) 
}
\end{equation}
and $\theta \circ \delta_{r} = \delta_{r}\circ \theta.$\\

\begin{theorem}
Let $p$ be a prime number, $p>2$. $PComS(p,\frac{p-1}{2},2a(a-p)+\frac{p(p-1)}{2})$ 
there exist in $\G_{p}(a)$ for all $p$ and $1<a<p$.
\end{theorem}
\begin{proof}
Pick $X$ in $\G_{p}(a)$. On the one hand, $\vert\Delta_{p}\vert=\phi(p)=p-1$. 
On the other hand, $\delta_{r}=\delta_{p-1}\delta_{p-r}$ implies that 
$\theta((\delta_{r}X)_{C})=\theta((\delta_{p-1}\delta_{p-r}X)_{C})=\theta((\delta_{p-r}X)_{C})$.
As $\theta(X_{C})$ has no fixed point under the action of $\delta_{r}\in\delta_{p}$, except 
$\mathsf{P}_{X}(0)$, then $\sum_{r=1}^{p-1}\mathsf{P}_{\delta_{r}X}(k)=\sum_{k=1}^{p-1}\mathsf{P}_{X}(k)=(2a-p)^{2}-p$. Hence
\begin{equation*}
c=\frac{1}{2}[(2a-p)^{2}-p]=2a(a-p)+\frac{p(p-1)}{2}.
\end{equation*}
\end{proof}

From the above, it is clear that the families $PComS(n,1,n-4)$ and $PComS(p,\frac{p-1}{2},2a(a-p)+\frac{p(p-1)}{2})$ are basic sets of the Schur rings $S_{C}$ and $S_{D}$, respectively.
Other relationship between families $PComS$ and Schur rings over $\Z_{2}^{n}$ will be shown 
in subsequent sections.

On ther other hand, in the theorem below we establish invariance properties of a $PComS$

\begin{theorem}
Let $\{A_{iC}\}$ be a family of $PComS(n,q,c)$. Then
\begin{enumerate}
\item $PComS(n,q,c)$ is invariant by negating any number of circulant sets in $\{A_{iC}\}$.
\item $PComS(n,q,c)$ is invariant by reversing any number of circulant sets in $\{A_{iC}\}$.
\end{enumerate}
\end{theorem}
\begin{proof}
This properties are derived from the properties of the autocorrelation function.
\end{proof}

\begin{definition}
Two families $PComS$ are \textbf{equivalent} if they are obtained via condictions in the 
above theorem.
\end{definition}

\begin{definition}
We say that a family $PComS(n,q,c)$ is \textbf{no trivial} if there are no families $PComS(n,q_{1},c_{1})$ and $PComS(n,q_{2},c_{2})$ such that $q=q_{1}+q_{2}$ y $c=c_{1}+c_{2}$. 
\end{definition}

The following non-equivalent families of $PComS$ are obtained by exhaustive search in 
$\mathbb{Z}_{2}^{n}$, with $n=4,5,6,7,8,9$.

\begin{theorem}
In $\mathbb{Z}_{2}^{4}$ the only non-trivial family of periodic compatible sequences is $PComS(4,1,0)=\{+---\}$.
\end{theorem}

\begin{theorem}\label{PComS_5}
In $\mathbb{Z}_{2}^{5}$ the only non-trivial and non-equivalent families of periodic compatible sequences are $PComS(5,1,1)$, $PComS(5,2,-2)$. See appendix.
\end{theorem}

\begin{theorem}
In $\mathbb{Z}_{2}^{6}$ the only non-trivial and non-equivalent families of periodic compatible sequences are $PComS(6,1,2)$, $PComS(6,3,-2)$. See appendix.
\end{theorem}

\begin{theorem}\label{PComS_7}
In $\mathbb{Z}_{2}^{7}$ the only non-trivial and non-equivalent families of periodic compatible sequences are $PComS(7,1,3)$, $PComS(7,3,1)$, $PComS(7,1,-1)$, $PComS(7,3,-3)$. See appendix.
\end{theorem}

\begin{theorem}
In $\mathbb{Z}_{2}^{8}$ the only non-trivial and non-equivalent families of periodic compatible sequences are $PComS(8,1,4)$, $PComS(8,2,0)$, $PComS(8,3,0)$, $PComS(8,4,0)$, $PComS(8,4,-4)$,
$PComS(8,5,4)$, $PComS(8,9,-8)$. See appendix.
\end{theorem}

\begin{theorem}\label{PComS_9}
In $\mathbb{Z}_{2}^{9}$ the only non-trivial and non-equivalent families of periodic compatible sequences are $PComS(9,1,5)$, $PComS(9,2,-2)$, $PComS(9,4,0)$, $PComS(9,4,8)$, 
$PComS(9,4,-4)$, $PComS(9,5,1)$, $PComS(9,6,6)$, $PComS(9,7,-1)$, $PComS(9,9,-3)$, 
$PComS(9,10,2)$, see appendix.
\end{theorem}

\section{Bounds on $PComS(n,q,c)$ in $\G_{nq}(a)$}

Let $PComS(n,q,c)=\{A_{1C},...,A_{qC}\}$. Define $X=(A_{1},A_{2},...,A_{q})$ for the rest of 
this paper. It is clear that $X\in\G_{nq}(a)$ for some $a$. In this section we will find an
upper bounds on the number of families $PComS(n,q,c)$ in $\G_{nq}(a)$. 

In [11] it proved that if there exist a circulant Hadamard matrix, then 
$\mathcal{N}_{X}(R_{2})\neq0$. The following lemma generalize the above result for a $PComS$
\begin{lemma}
In a $PComS(n,q,c)=\{A_{1C},...,A_{qC}\}\neq PComS(n,1,n-4)$ always hold that $\sum_{i=1}^{q}\mathcal{N}_{A_{i}}(R_{2})\neq0$.
\end{lemma}
\begin{proof}
Suppose $\mathcal{N}_{X}(R_{2})=0$. Then $\sum_{j\geq3}\mathcal{N}_{X}(R_{j})=\frac{l(X)}{2}$. 
As $\mathcal{N}_{X}(R_{2})=\mathcal{N}_{X}(R_{1}R_{1})$ for families $PComS$, then $X$ has no
contains run string $R_{3}, R_{4},...$. Hence is only possible $PComS(n,q,c)=PComS(n,1,n-4)$, 
a contradiccion.
\end{proof}

Now, we characterize the number of 1's that could have the composition of $X$ if is in 
partitioned form
\begin{lemma}
Let $X=P\diamond_{l(X)/2}Q$, with $P\in C_{l(X)/2}(a)$ y $Q\in C_{l(X)/2}(nq-a)$. Then
\begin{eqnarray}
l(X)-a\leq&\mathcal{N}_{P}(R_{1})&\leq\frac{l(X)}{2}-1\\
1\leq&\mathcal{N}_{Q}(R_{1})&\leq a-\frac{l(X)}{2}.
\end{eqnarray}
From here can be deduced that $l(X)/2<a<l(X)$. 
\end{lemma}
\begin{proof}
Let $k$ be the minimum number of 1's in any composition $A$ in $C_{l(X)/2}(a)$. Then $A$
has $l(X)/2-k$ summands larger than 1. As $k$ is minimum, all summands langer than 1 must be
exactly 2. Thus $2(l(X)/2-k)=a-k$ has solution $k=l(X)-a$. By above lemma, since $\mathcal{N}_{X}(R_{2})\neq0$, it is not possible $k=0$. Hence $k>0$. By the same argument, it is not possible 
$\mathcal{N}_{P}(R_{1})=\frac{l(X)}{2}$. The remaining results follows from $\mathcal{N}_{Q}(R_{1})=\frac{l(X)}{2}-\mathcal{N}_{P}(R_{1})$.
\end{proof}

For the next theorem we need some notation. Let $p(n)=1k_{1}+2k_{2}+\cdots rk_{r}$
be any partition of the positive integer $n$, where $k_{j}$ denotes the number of parts of size
$j$ in the partition of $n$ and $k_{1}+k_{2}+\cdots+k_{r}$ denotes the length of the partition
$p(n)$, thus, the number of summands contained in $p(n)$.

\begin{theorem}\label{theorem_approx}
The number of $PComS(n,q,c)$ in $\G_{nq}(a)$ is bounded from above by
\begin{equation}\label{approximation}
B(n,q,c,a)=\sum_{h=l(X)-a}^{\frac{l(X)}{2} -1}\binom{\frac{l(X)}{2}}{h,i_{2},...,i_{r}}\binom{\frac{l(X)}{2}}{(\frac{l(X)}{2}-h),j_{2},...,j_{s}}
\end{equation}
with $p(a)=1h+2i_{2}+\cdots+ri_{r}$ and $i_{2}+\cdots+i_{r}=\frac{l(X)}{2}-h$ and with
$p(nq-a)=1(\frac{l(X)}{2}-h)+2j_{2}+\cdots+sj_{s}$ y $j_{2}+\cdots+j_{s}=h$.
\end{theorem}
\begin{proof}
Pick $X=P\diamond_{l(X)/2}Q$ in $\G_{nq}(a)$ with $P\in C_{l(X)/2}(a)$ and $Q\in C_{l(X)/2}(nq-a)$.
Since the number of 1's in $X$ is $\frac{l(X)}{2}$, then el number of 1's in $P$ and $Q$ are $h$ and 
$\frac{l(X)}{2}-h$, respectively. Then $P$ is the combination of $h$ 1's and each
composition of $a-h$ with length $\frac{l(X)}{2}-h$ and $Q$ is the combination of $l(X)-a-h$ 
1's and each composition of $nq-a-\frac{l(X)}{2}+h$ with length $h$, with summands in both composition taken from the set $A=\{2,3,4,...\}$. Thus, we have
\begin{equation*}
\binom{\frac{l(X)}{2}}{h,i_{2},...,i_{r}}
\end{equation*}
possible combinations for $P$ and
\begin{equation*}
\binom{\frac{l(X)}{2}}{(\frac{l(X)}{2}-h),j_{2},...,j_{s}}
\end{equation*}
possible combinationes for $Q$. By above lemma $h$ ranges between $l(X)-a$ and $\frac{l(X)}{2}-1$ 
and so we have (\ref{approximation}).
\end{proof}

\section{Example of families of $PComS$}

In this section examples of families $PComS$ will be shown. Within these examples, we have 
Hadamard matrices of type circulant, with one and two cores, of type Goethals-Seidel 
and partials also perfect binary sequences. We will use the theorem \ref{theorem_approx} 
to obtain upper bounds on such families. 

\subsection{Hadamard matrices}

A Hadamard matrix $H$ is a $n\times n$ matrix all of whose entries are $+1$ or 
$-1$ which satisfies $HH^{t}=nI_{n}$, where $H^{t}$ is the transpose of $H$
and $I_{n}$ is the unit matrix of order $n$. It is also known that, if a
Hadamard matrix of order $n>1$ exists, $n$ must have the value $2$ or be
divisible by 4. It has been conjecture that this condition also insures the
existence of a Hadamard matrix. As example of Hadamard matrices we have

\subsubsection{Circulant Hadamard matrices}

A circulant Hadamard matrix of order $n$ is a square matrix of the form 
\begin{equation}
H =
\left(\begin{array}{cccc}
	a_{1} & a_{2} & \cdots & a_{n}\\
	a_{n} & a_{1} & \cdots & a_{n-1}\\
	\cdots & \cdots & \cdots & \cdots\\
	a_{2} & a_{3} & \cdots & a_{1}
\end{array}\right)
\end{equation}
No circulant Hadamard matrix of order larger than 4 has ever been found. This let the following

\begin{conjecture}
No circulant Hadamard matrix of order larger than 4 exists.
\end{conjecture}

We know that if a circulant Hadamard matrix exists, then their order must be $4m^{2}$, i.e., 
a circulant Hadamard matrix is a $PComS(4m^{2},1,0)$ in $\G_{4m^{2}}(2m^{2}-m)$. Then 
$l(X)=2m^{2}$ and $\mathcal{N}_{X}(R_{1})=m^{2}$ and from the lemma 2
\begin{eqnarray}
m\leq&\mathcal{N}_{P}(R_{1})&\leq m^{2}-1\\
1\leq&\mathcal{N}_{Q}(R_{1})&\leq m^{2}-m
\end{eqnarray}
and the number of circulant Hadamard matrices is bounded from above by
\begin{equation}\label{approximation_cir_had}
B(4m^{2},1,0,2m^{2}-m)=\sum_{h=m}^{m^{2}-1}\binom{m^{2}}{h,i_{2},...,i_{r}}\binom{m^{2}}{(m^{2}-h),j_{2},...,j_{s}}
\end{equation}
with $p(2m^{2}-m)=1h+2i_{2}+\cdots+ri_{r}$ and $i_{2}+\cdots+i_{r}=m^{2}-h$ and with
$p(2m^{2}+m)=1(m^{2}-h)+2j_{2}+\cdots+sj_{s}$ and $j_{2}+\cdots+j_{s}=h$.

From [10] it is known that if a circulant Hadamard matrix $H$ exists, then
\begin{equation}
H\subset\G_{2m^{2}}(a)\times\G_{2m^{2}}(b)
\end{equation}
where $\frac{m^{2}-m}{2}\leq a\leq \frac{3m^{2}-m}{2}$ and $b=2m^{2}-m-a$.

We will use the above result to improve (\ref{approximation_cir_had}).

\begin{theorem}
The number of circulant Hadamard matrices in $\G_{2m^{2}}(a)\times\G_{2m^{2}}(b)$ with
$\frac{m^{2}-m}{2}\leq a\leq \frac{3m^{2}-m}{2}$ and $b=2m^{2}-m-a$ is bounded from above by
\begin{eqnarray}
\sum_{h=m}^{m^{2}-1}\sum_{a=\frac{m^{2}-m}{2}}^{\frac{3m^{2}-m}{2}}&&\sum_{d=0}^{h}\sum_{a_{1}=0}^{a}\binom{a_{1}}{d,i_{2},...,i_{r}}\binom{a-a_{1}}{(h-d),j_{2},...,j_{s}}\nonumber\\
&\times&\sum_{e=0}^{m^{2}-h}\sum_{b_{1}=0}^{2m^{2}-m-a}\binom{b_{1}}{e,k_{2},...,k_{t}}\binom{2m^{2}-m-a-b_{1}}{(m^{2}-h-e),l_{2},...,l_{w}}
\end{eqnarray}
where $i_{2}+\cdots+i_{r}=a_{1}-d$, $i_{g}\geq2$, $j_{2}+\cdots+j_{s}=a-a_{1}-h+d$, $j_{g}\geq2$ and
$k_{2}+\cdots+k_{t}=b_{1}-e$, $k_{g}\geq2$, $l_{2}+\cdots+l_{w}=m^{2}-m-a-b_{1}+h+e$, $l_{g}\geq2$.
\end{theorem}
\begin{proof}
Let $X=(A,B)$ with $A\in\G_{2m^{2}}(a)$ and $B\in\G_{2m^{2}}(b)$. As $\mathcal{N}_{P}(R_{1})$ ranges
in $[m,m^{2}-1]$, then the number of run string $+$ in $A$ is a positive integer $d$ that ranges
in $[0,h]$ provided that $h$ is in $[m,m^{2}-1]$. Then $h-d$ is the number of run string $+$ in 
$B$. Also, we know that $a$ is in $[\frac{m^{2}-m}{2},\frac{3m^{2}-m}{2}]$, therefore the 
number of compositions of $P$ of length $m^{2}$ with number of run string 
$\mathcal{N}_{P}(R_{1})$ is
\begin{equation}
\binom{a_{1}}{d,i_{2},...,i_{r}}\binom{a-a_{1}}{(h-d),j_{2},...,j_{s}}
\end{equation}
where $i_{2}+\cdots+i_{r}=a_{1}-d$, $i_{g}\geq2$, $j_{2}+\cdots+j_{s}=a-a_{1}-h+d$, $j_{g}\geq2$,
con $a_{1}\in[0,a]$.
Equaly, the number of run string $-$ in $A$ is a positive integer $e$ that ranges in 
$[0,m^{2}-h]$ and $m^{2}-h-e$ is the number of run string $-$ in $B$. As $b=2m^{2}-m-a$, then
the number of compositions of $Q$ of length $m^{2}$ with number of run string 
$\mathcal{N}_{Q}(R_{1})$ is
\begin{equation}
\binom{b_{1}}{e,k_{2},...,k_{t}}\binom{2m^{2}-m-a-b_{1}}{(m^{2}-h-e),l_{2},...,l_{w}}
\end{equation}
where $k_{2}+\cdots+k_{t}=b_{1}-e$, $k_{g}\geq2$, $l_{2}+\cdots+l_{w}=m^{2}-m-a-b_{1}+h+e$, 
$l_{g}\geq2$ y $b_{1}\in[0,2m^{2}-m-a]$. The expected bound is obtained from adding (7.7) and 
(7.8) for $h,a,d,a_{1},e,b_{1}$.
\end{proof}

\subsubsection{Hadamard matrices with one circulant core}

An Hadamard matrix with one circulant core (see [15]) of order $p+1$ is a $(p+1)\times(p+1)$ 
matrix of the form
\begin{equation*}
\left( 
\begin{array}{cc}
1 & e \\ 
e^{t} & X_{C}
\end{array}
\right)
\end{equation*}
where $e$ is the row vector $(1,1,1,...,1)$ of dimension $p$ and $e^{t}$
the transposed vector of $e$ and $X_{C}=(x_{ij})$ a circulant matrix or circulant core of order $p$.
An Hadamard matrix of order $p+1$ with circulant core can be constructed if
\begin{enumerate}
\item $p\equiv3\mod 4 $ is a prime;
\item $p=q(q+2)$ where $q$ and $q+2$ are both primes;
\item $p=2^{t}-1$\ where $t$ is a positive integer;
\item $p=4x^{2}+27$ where $p$ is a prime and $x$ a positive integer.
\end{enumerate}
\begin{conjecture}
If a Hadamard matrix with one circulant core of order $p+1$ exists, then $p$ is as above.
\end{conjecture}

Is clear that $X$ is in $\G_{p}(\frac{p-1}{2})$. Also $X$ is a $PComS(p,1,-1)$. Then 
$l(X)=\frac{p+1}{2}$, $\mathcal{N}_{X}(R_{1})=\frac{p+1}{4}$ and from lemma 2
\begin{eqnarray}
1\leq&\mathcal{N}_{P}(R_{1})&\leq\frac{p-3}{4}\\
1\leq&\mathcal{N}_{Q}(R_{1})&\leq\frac{p-3}{4}.
\end{eqnarray}

We can to improve the above inequalities with the following
\begin{lemma}
$1<\mathcal{N}_{P}(R_{1})<\frac{p-3}{4}$, $p>11$.
\end{lemma}
\begin{proof}
If $\mathcal{N}_{P}(R_{1})=1$, then $\mathcal{N}_{Q}(R_{1})=\frac{p-3}{4}$. Hence $P$ contains
all compositions with one 1 and $\frac{p-3}{2}$ 2's and $Q$ contains all compositions with 
$\frac{p-3}{2}$ 1's and one summand iqual to $\frac{p+5}{4}$. Thus the maximum run of 1's in 
$P\diamond_{(p+1)/4}Q$ is $-+-$. Therefore $\mathcal{N}_{X}(R_{2})=2$. But this number is less that 
$\frac{p-3}{4}$ for $p>11$. The proof is equal for $\mathcal{N}_{P}(R_{1})=\frac{p-3}{4}$.
\end{proof}

From the theorem \ref{theorem_approx} we have

\begin{corollary}
The number of $PComS(p,1,-1)$ in $\G_{p}(\frac{p-1}{2})$, with $p>11$, is bounded from above by
\begin{equation}\label{approximation_one_core}
B\left(p,1,-1,\frac{p-1}{2}\right)=\sum_{h=2}^{\frac{p-3}{4}}\binom{\frac{p+1}{4}}{h,i_{2},...,i_{r}}\binom{\frac{p+1}{4}}{(\frac{p+1}{4}-h),j_{2},...,j_{s}}
\end{equation}
with $p(\frac{p-1}{2})=1h+2i_{2}+\cdots+ri_{r}$ and $i_{2}+\cdots+i_{r}=\frac{p+1}{4}-h$ and with
$p(\frac{p+1}{2})=1(\frac{p+1}{4}-h)+2j_{2}+\cdots+sj_{s}$ and $j_{2}+\cdots+j_{s}=h$.
\end{corollary}

\subsubsection{Hadamard matrices with two circulant cores}

An Hadamard matrix of order $2n+2$ with two circulant cores(see [16]) have the following array
\begin{equation*}
H_{2n+2}=\left( 
\begin{array}{cccccccc}
+ &+ & + & \cdots & + & + & \cdots & + \\ 
+ & - & + & \cdots & + & - & \cdots & - \\ 
+ & + &  &  &  &  &  &  \\ 
\vdots & \vdots &  & A&  &  & B&  \\ 
+ & + &  &  &  &  &  &  \\ 
+ & - &  &  &  &  &  &  \\ 
\vdots & \vdots &  & B^{t} &  &  & -A^{t} &  \\ 
+ & - &  &  &  &  &  & 
\end{array}
\right)
\end{equation*}
where $A=(a_{ij})$, $B=(b_{ij})$ are two circulant matrices of order $n$,
i.e. $a_{ij}=a_{1,(j-i+1)\mod n}$, $b_{ij}=b_{1,(j-i+1)\mod n}$, is said to have two circulant cores.
The two circulant matrices $A$ and $B$ satisfy the matrix equation
\begin{equation}\label{-2}
AA^{t}+BB^{t}=(2n+2)I_{n}-2J_{n}
\end{equation}
where $I_{n}$ is the identity matrix or order $n$ and $J_{n}$ is a matrix of order $n$ whose elements are all equal to $1.$
This matrices can be constructed if
\begin{enumerate}
\item $n$ is a prime;
\item $2n+1$ is a prime power;
\item $n=2^{k}-1$, $k\geq 2$;
\item $n=p(p+2)$ where $p$ and $p+2$ are both primes;
\item $n=3,5,\ldots ,63,$ and
\item $n=143.$
\end{enumerate}
\begin{conjecture}
For every odd $n$ there exists a Hadamard matrix of order $2n+2$ with two circulant cores.
\end{conjecture}
The rows first of $A$ and $B$ is in $\G_{n}\left(\frac{n-1}{2}\right)$. The matrices $A,B$ form a family $PComS(n,2,-2)$ in $\G_{2n}(n-1)$ with $l(A)+l(B)=n+1$, $\mathcal{N}_{A}(R_{1})+\mathcal{N}_{B}(R_{1})=\frac{n+1}{2}$ and from the lemma 2, with $A=P_{1}\diamond_{r}Q_{1}$ and $B=P_{2}\diamond_{s}Q_{2}$, $r+s=n+1$
\begin{eqnarray}
2\leq&\mathcal{N}_{P_{1}}(R_{1})+\mathcal{N}_{P_{2}}(R_{1})&\leq\frac{n-1}{2}\\
1\leq&\mathcal{N}_{Q_{1}}(R_{1})+\mathcal{N}_{Q_{2}}(R_{1})&\leq\frac{n-3}{2}.
\end{eqnarray}
Then by theorem \ref{theorem_approx} the number of $PComS(n,2,-2)$ in $\G_{2n}(n-1)$ is bounded
from above by
\begin{equation}
B(n,2,-2,n-1)=\sum_{h=2}^{\frac{n-1}{2}}\binom{\frac{n+1}{2}}{h,i_{2},...,i_{r}}\binom{\frac{n+1}{2}}{(\frac{n+1}{2}-h),j_{2},...,j_{s}}
\end{equation}
with $p(n-1)=1h+2i_{2}+\cdots+ri_{r}$ and $i_{2}+\cdots+i_{r}=\frac{n+1}{2}-h$ and with
$p(n+1)=1(\frac{n+1}{2}-h)+2j_{2}+\cdots+sj_{s}$ and $j_{2}+\cdots+j_{s}=h$.

\subsubsection{Hadamard matrices of Goethals-Seidel type}

An Hadamard matrix of order $4n$, $n$ odd, is Goethals-Seidel type if it has the following
arrangement, $GS$-array, 
\begin{equation*}
H=\left( 
\begin{array}{cccc}
A & BR & CR & DR \\ 
-BR & A & -D^{t}R & C^{t}R \\ 
-CR & D^{t}R & A & -B^{t}R \\ 
-DR & -C^{t}R & B^{t}R & A
\end{array}
\right)
\end{equation*}
where $R$ will be the back-circulant identity matrix of order $n$,
\begin{equation*}
R=\left( 
\begin{array}{ccccc}
0&0&\cdots&0&1 \\ 
0&0&\cdots&1&0 \\ 
0&1&\cdots&0&0 \\ 
1&0&\cdots&0&0 
\end{array}
\right)
\end{equation*}
such that 
\begin{equation}\label{array_GS}
AA^{t}+BB^{t}+CC^{t}+DD^{t}=4nI_{n}.
\end{equation}
An Hadamard matrix $H$ is \textit{skew} if $HH^{t}=4nI_{n}$ and $H+H^{t}=2I_{n}$. A known result 
on skew-Hadamard matrices can be found in [19]. If $A,B,C,D$ are square circulant
matrices of order $n$, if $A$ is skew, and if 
\begin{equation}
AA^{t}+BB^{t}+CC^{t}+DD^{t}=(4n-1)I_{n}
\end{equation}
then
\begin{equation*}
H=\left( 
\begin{array}{cccc}
A+I_{n} & BR & CR & DR \\ 
-BR & A+I_{n} & -D^{t}R & C^{t}R \\ 
-CR & D^{t}R & A+I_{n} & -B^{t}R \\ 
-DR & -C^{t}R & B^{t}R & A+I_{n}
\end{array}
\right)
\end{equation*}
is an skew Hadamard matrix.

On the other hand, two matrices $M,N$ of order $n$ are amicable if $MN^{t}=NM^{t}$. Let $A,B,C,D$ be
$\mathbb{Z}_{2}$-matrices of order $n$. We say that $A,B,C,D$ are \textit{Williamson type} matrices
if they are pairwise amicable and satisfiy equation (\ref{array_GS}). In [18],[20],[21] we can see 
more on Hadamard matrices of Goethals-Seidel type and Williamson type.

\begin{conjecture}
If $n$ is a multiple of $4$, then a Goethals-Seidel Hadamard matrix of order $n$ exists.
\end{conjecture}

It is a fact that if a $GS$-array exists, then there exist integer numbers $a,b,c,d$ such that
\begin{equation}
a^{2}+b^{2}+c^{2}+d^{2}=4n.
\end{equation}
If the matrices $A,B,C,D$ have row sums $a,b,c,d$, all positive numbers, respectively, then 
$A,B,C,D$ form a family $PComS(n,4,0)$ in $\G_{4n}(2n-\frac{a+b+c+d}{2})$ with $A\subset\G_{n}(\frac{n-a}{2})$, $B\subset\G_{n}(\frac{n-b}{2})$, $C\subset\G_{n}(\frac{n-c}{2})$, $D\subset\G_{n}(\frac{n-d}{2})$, with $l(A)+l(B)+l(C)+l(D)=2n$, $\mathcal{N}_{A}(R_{1})+\mathcal{N}_{B}(R_{1})+\mathcal{N}_{C}(R_{1})+\mathcal{N}_{D}(R_{1})=n$ and from the lemma 2, with $A=P_{1}\diamond_{r}Q_{1}$, $B=P_{2}\diamond_{s}Q_{2}$, $C=P_{3}\diamond_{t}Q_{3}$, $D=P_{4}\diamond_{w}Q_{4}$, $r+s+t+w=2n$,
\begin{eqnarray}
\frac{a+b+c+d}{2}\leq&\sum_{i=1}^{4}\mathcal{N}_{P_{i}}(R_{1})&\leq n-1\\
1\leq&\sum_{i=1}^{4}\mathcal{N}_{Q_{i}}(R_{1})&\leq n-\frac{a+b+c+d}{2}.
\end{eqnarray}
Then, from the theorem \ref{theorem_approx} the number of $PComS(n,4,0)$ in $\G_{4n}(2n-\frac{a+b+c+d}{2})$ is bounded from above by
\begin{equation}
B\left(n,4,0,2n-\frac{a+b+c+d}{2}\right)=\sum_{h=\frac{a+b+c+d}{2}}^{n-1}\binom{n}{h,i_{2},...,i_{r}}\binom{n}{(n-h),j_{2},...,j_{s}}
\end{equation}
with $p\left(2n-\frac{a+b+c+d}{2}\right)=1h+2i_{2}+\cdots+ri_{r}$ and $i_{2}+\cdots+i_{r}=n-h$ and with $p\left(2n+\frac{a+b+c+d}{2}\right)=1(n-h)+2j_{2}+\cdots+sj_{s}$ and $j_{2}+\cdots+j_{s}=h$.

\subsection{Partial Hadamard Matrices}

A partial Hadamard matrix denoted $PH(k\times n)$ is a $k\times n$ matrix all of whose entries are 
$+1$ or $-1$ which satisfies $PH\cdot PH^{t}=nI_{k}$. For more details can see [22],[23],[24].

On the other hand, from theorem 4 in [9] we can conclude that if a partial Hadamard matrix $PH$
exist, then there is either in a complete maximal $S$-sets even $[\E_{4n}]$ or in a complete 
maximal $S$-sets odd $[\Odd_{4n}]$. A circulant partial Hadamard matrix never is a family
$PcomS$, however the theorem \ref{theorem_approx} provides a upper bound for all circulant partial
Hadamard matrices $PH$ of order $k\times4n$, with $k\geq2$.

The following theorem shows the relationship between $PComS$ and partial Hadamard matrices

\begin{theorem}
If there exist $PComS(n,q,-c)$, $c\geq0$, then there exist a partial Hadamard matrix 
$PH(n\times(nq+c))$.
\end{theorem}
\begin{proof}
Let $\{A_{iC}\}$ be a family $PComS(n,q,-c)$ in $\mathbb{Z}_{2C}^{n}$. Define the
$n\times(nq+c)$ matrix $PH=\left(\overset{c}{\overbrace{\textbf{e}_{n}^{t},...,\textbf{e}_{n}^{t}}},A_{1C},...,A_{qC}\right)$, where $\textbf{e}_{n}$ is the row vector with $n$ 1's. Then
it is easy to see that $PH\cdot PH^{t}=c\cdot\textbf{e}_{n}^{t}\textbf{e}_{n}+\sum_{i=1}^{q}A_{iC}^{2}=(nq+c)I_{n}$ and therefore $H$ is a partial Hadamard matrix.
\end{proof}

\begin{example}
From theorems \ref{PComS_5} to \ref{PComS_9} we obtain the following partial Hadamard matrices:
\begin{eqnarray*}
PH_{1}&=&\left(\mathbf{e}_{5}^{t},\mathbf{e}_{5}^{t},(++---)_{C},(+--+-)_{C}\right)_{5\times12}\\
PH_{2}&=&\left(\mathbf{e}_{6}^{t},\mathbf{e}_{6}^{t},(++----)_{C},(+-+---)_{C},(++-+--)_{C}\right)_{6\times20}\\
PH_{3}&=&\left(\mathbf{e}_{7}^{t},,(++-+---)_{C}\right)_{7\times8}\\
PH_{4}&=&\left(\mathbf{e}_{7}^{t},\mathbf{e}_{7}^{t},\mathbf{e}_{7}^{t},(+++----)_{C},(++--+--)_{C},(+-+-+--)_{C}\right)_{7\times24}\\
PH_{5}&=&((+-+-----)_{C},(+++--+--)_{C})_{8\times16}\\
PH_{6}&=&((+-+-----)_{C},(++-+----)_{C},(++--+---)_{C})_{8\times24}\\
PH_{7}&=&((++------)_{C},(+--+----)_{C},(+++-+---)_{C},\\
PH_{8}&=&\ (+--++-+-)_{C})_{8\times32}\\
PH_{9}&=&(\mathbf{e}_{9}^{t},\mathbf{e}_{9}^{t},(+++-+----)_{C},(+-+--++--)_{C})_{9\times20}\\
PH_{10}&=&(\mathbf{e}_{9}^{t},\mathbf{e}_{9}^{t},\mathbf{e}_{9}^{t},\mathbf{e}_{9}^{t},(++++-----)_{C},(++---++--)_{C},\\
&&\ (+--+-+-+-)_{C},(++-+--+--)_{C})_{9\times40}\\
PH_{11}&=&(\mathbf{e}_{9}^{t},(+--+-----)_{C},(+-+--++--)_{C},(++-+---+-)_{C},\\
&&\ (++----++-)_{C},(+++------)_{C},(++---+---)_{C},\\
&&\ (+----+-+-)_{C})_{9\times64}\\
PH_{12}&=&(\mathbf{e}_{9}^{t},\mathbf{e}_{9}^{t},\mathbf{e}_{9}^{t},(+++------)_{C}
,(++---+---)_{C},\\
&&\ (+----+-+-)_{C},(++-+-----)_{C},(+----++--)_{C},\\
&&\ (+--+---+-)_{C},(+-+--++--)_{C},(++-+---+-)_{C},\\
&&\ (++----++-)_{C})_{9\times84}
\end{eqnarray*}
\end{example}

\begin{theorem}
If there exist $PComS(n,q,c_{1})$ and $PComS(n,q,c_{2})$, with $c_{1}+c_{2}=-2$, then there
exist a partial Hadamard matrix $PH(2n\times2(nq+1))$.
\end{theorem}
\begin{proof}
Let $PComS(n,q,c_{1})=\{A_{1C},...,A_{qC}\}$ and $PComS(n,q,c_{2})=\{B_{1C},...,B_{qC}\}$ be families
$PComS$ with $c_{1}+c_{2}=-2$. Define the matrix $2n\times(2nq+2)$
\begin{equation}
PH=
\left(
\begin{array}{cccccccc}
\mathbf{e}_{n}^{t}&\mathbf{e}_{n}^{t}&A_{1C}&\cdots&A_{qC}&B_{1C}\cdots&B_{qC}\\
\mathbf{e}_{n}^{t}&-\mathbf{e}_{n}^{t}&B_{1C}&\cdots&B_{qC}&-A_{1C}\cdots&-A_{qC}
\end{array}
\right).
\end{equation}
Then
\begin{eqnarray*}
PH\cdot PH^{t} &=&
\left(
\begin{array}{cc}
2\mathbf{e}_{n}^{t}\mathbf{e}_{n}+\sum_{i=1}^{q}(A_{iC}^{2}+B_{iC}^{2})&\mathbf{0}\\
\mathbf{0}&2\mathbf{e}_{n}^{t}\mathbf{e}_{n}+\sum_{i=1}^{q}(A_{iC}^{2}+B_{iC}^{2})
\end{array}
\right)\\
&=&\left(
\begin{array}{cc}
(2nq+2)I_{n}&\mathbf{0}\\
\mathbf{0}&(2nq+2)I_{n}
\end{array}
\right)\\
&=&2(nq+1)I_{2n}
\end{eqnarray*}
\end{proof}

\begin{example}
From the theorems \ref{PComS_7} and \ref{PComS_9} we obtain the following partial Hadamard matrices
\begin{equation*}
PH_{13}=
\left(
\begin{array}{cccccccc}
\mathbf{e}_{7}^{t}&\mathbf{e}_{7}^{t}&A_{1C}&A_{2C}&A_{3C}&B_{1C}&B_{2C}&B_{3C}\\
\mathbf{e}_{7}^{t}&-\mathbf{e}_{7}^{t}&B_{1C}&B_{2C}&B_{3C}&-A_{1C}&-A_{2C}&-A_{3C}\\
\end{array}
\right)_{14\times44}
\end{equation*}
where $A_{1}=++-----$, $A_{2}=+-+----$, $A_{3}=+---+--$, $B_{1}=+++----$,
$B_{2}=++--+--$, $B_{3}=+-+-+--$,
and
\begin{equation*}
PH_{14}=
\left(
\begin{array}{cccc}
\mathbf{e}_{9}^{t}&\mathbf{e}_{9}^{t}&PComS(9,7,-1)&PComS(9,7,-1)\\
\mathbf{e}_{9}^{t}&-\mathbf{e}_{9}^{t}&PComS(9,7,-1)&-PComS(9,7,-1)
\end{array}
\right)_{18\times128}
\end{equation*}
where $PComS(9,7,-1)$ is like in the Appedix.
\end{example}

\subsection{Perfect Binary Sequences}

Let $\theta^{*}$ be the restriction of $\theta$ defined by $\theta^{*}(X_{C})=(\mathsf{P}_{X}(1),..,\mathsf{P}_{X}(n-1))$. It is know that $\sum_{k=1}^{n-1}\mathsf{P}_{X}(k)=(2a-n)^{2}-n$ if 
$X_{C}\subset\G_{n}(a)$, therefore $\theta^{*}$ sends the plane $\G_{n}(a)$ to plane
$X_{1}+X_{2}+\cdots+X_{n-1}=(2a-n)^{2}-n$. On the other hand, a binary sequence with 2-level
autocorrelation values is called \textit{perfect} if the nontrivial autocorrelation values $d$ are 
as small as possible in absolute value. If $X\in\G_{n}(a)$ is a perfect binary sequence, 
then $(2a-n)^{2}-n=(n-1)d$. Solving for $a$ we obtain
\begin{equation}
a=\frac{n\pm\sqrt{(n-1)d+n}}{2}.
\end{equation}
Thus for $d\in\{-2,-1,0,1,2\}$ we have
\begin{enumerate}
\item $a=\frac{n\pm\sqrt{n}}{2}$, con $d=0$ y $n\equiv0\mod4$.
\item $a=\frac{n\pm\sqrt{2n-1}}{2}$, con $d=1$ y $n\equiv1\mod4$.
\item $a=\frac{n\pm\sqrt{2-n}}{2}$, con $d=-2$ y $n\equiv2\mod4$.
\item $a=\frac{n\pm\sqrt{3n-2}}{2}$, con $d=2$ y $n\equiv2\mod4$.
\item $a=\frac{n\pm1}{2}$, con $d=-1$ y $n\equiv3\mod4$.
\end{enumerate}

The above agrees with the five different class of cyclic difference sets corresponding 
to the perfect binary sequences given by Jungnickel and Pott in [25]. The cases $d=0$ and $d=-1$
correspond to Hadamard matrices and were discussed in above sections. Therefore, we only have to 
analyse the cases $d=-2,1,2$.

\subsubsection{Case $n\equiv1\mod4$}

If there exist a perfect binary sequence with $d=1$, then this is a family $PComS(n,1,1)$ with
$n\equiv1\mod4$. As $PcomS(n,1,n-4)$ exists for all $n$, it is trivial that $PComS(5,1,1)=\{+----\}$
is a perfect sequence. On the other hand, $a=\frac{n-\sqrt{2n-1}}{2}$ has integer solution if
$n=2u^{2}+2u+1$ and then $a=u^{2}$. If $X$ is a perfect binary sequence, $l(X)=u^{2}+u$ and 
$\mathcal{N}_{X}(R_{1})=\frac{u^{2}+u}{2}$. From lemma 2, with $X=P\diamond_{l(X)/2}Q$,
\begin{eqnarray*}
u\leq&\mathcal{N}_{P}(R_{1})&\leq\frac{u^{2}+u}{2}-1\\
1\leq&\mathcal{N}_{Q}(R_{1})&\leq\frac{u^{2}-u}{2}
\end{eqnarray*}
The equation $\frac{u^{2}+u}{2}-1=u$ has positive solution $u=2$. Then $n=13$, $l(X)=6$, 
$\mathcal{N}_{X}(R_{1})=3$, $\mathcal{N}_{P}(R_{1})=2$ and $\mathcal{N}_{Q}(R_{1})=1$. 
The run string such that $X$ is perfect is $(1,1,2,3,1,5)=+-++---+-----$. The families 
$PComS(5,1,1)$ and $PComS(13,1,1)$ are the only perfect binary sequences known. 
\begin{conjecture}
There is no families $PComS(2u^{2}+2u+1,1,1)$ with $u\geq3$.
\end{conjecture}

From theorem \ref{theorem_approx}, the number of $PComS(2u^{2}+2u+1,1,1)$ in 
$\G_{2u^{2}+2u+1}(u^{2})$ is bounded from above by
\begin{equation}
\sum_{h=u}^{\frac{u^{2}+u}{2} -1}\binom{\frac{u^{2}+u}{2}}{h,i_{2},...,i_{r}}\binom{\frac{u^{2}+u}{2}}{(\frac{u^{2}+u}{2}-h),j_{2},...,j_{s}}
\end{equation}
with $p(u^{2})=1h+2i_{2}+\cdots+ri_{r}$ and $i_{2}+\cdots+i_{r}=\frac{u^{2}+u}{2}-h$ and with
$p((u+1)^{2})=1(\frac{u^{2}+u}{2}-h)+2j_{2}+\cdots+sj_{s}$ and $j_{2}+\cdots+j_{s}=h$.

\subsubsection{Case $n\equiv2\mod4$}
For $d=-2$, $a=\frac{n-\sqrt{2-n}}{2}$ has integer solution for $n=2$. Equally $PComS(n,1,n-4)$
is a perfect binary sequence with $d=-2$ only if $n=2$. Thus $+-$ is the desired sequence. In fact 
the only.

If $d=2$, $PComS(n,1,n-4)$ is a perfect binay sequence only if $n=6$. This is the only known case.
\begin{conjecture}
There is no $PComS(n,1,2)$ with $n>6$. 
\end{conjecture}

Now, we will give bounds on the number of $PComS(n,1,2)$ in $\G_{n}(a)$ with $a=\frac{n-\sqrt{3n-2}}{2}$. First, note that $a$ has two integer solutions
\begin{enumerate}
\item $n=12u^{2}-16u+6$ y $a=6u^{2}-11u+5$,
\item $n=12u^{2}-4u+2$ y $a=6u^{2}-5u+2$,
\end{enumerate}
with $u\geq0$. If $X$ is a perfect binary sequence, then by lemma 2,
respectively
\begin{eqnarray*}
&l(X)&=6u^{2}+8u+2,\\
&\mathcal{N}_{X}(R_{1})&=3u^{2}-4u+1,\\
3u-3\leq&\mathcal{N}_{P}(R_{1})&\leq3u^{2}-4u,\\
1\leq&\mathcal{N}_{Q}(R_{1})&\leq3u^{2}-7u+4,
\end{eqnarray*}
and
\begin{eqnarray*}
&l(X)&=6u^{2}-2u,\\
&\mathcal{N}_{X}(R_{1})&=3u^{2}-2u,\\
3u-2\leq&\mathcal{N}_{P}(R_{1})&\leq3u^{2}-u+1,\\
1\leq&\mathcal{N}_{Q}(R_{1})&\leq3u^{2}-4u+2.
\end{eqnarray*}

In order to improve the bound from the theorem \ref{theorem_approx} we will use
the results below whose proofs are analogous to the proof the theorem 5 in [10]

\begin{theorem}
A family $PComS(12u^{2}-16u+6,1,2)$ can exist in $\G_{6u^{2}-8u+3}(a)\times\G_{6u^{2}-8u+3}(b)$ with 
$\frac{3u^{2}-7u+4}{2}\leq a\leq\frac{9u^{2}-15u+2}{2}$ and $b=6u^{2}-11u+5-a$.
\end{theorem}

\begin{theorem}
A family $PComS(12u^{2}-4u+2,1,2)$ can exist in $\G_{6u^{2}-2u+1}(a)\times\G_{6u^{2}-2u+1}(b)$ with 
$\lfloor\frac{3u^{2}-4u+3}{2}\rfloor\leq a\leq\lfloor\frac{3u^{2}+2u}{2}\rfloor$ and 
$b=6u^{2}-5u+2-a$.
\end{theorem}

Just as we did with circulant Hadamard matrices, we obtain upper bounds for families 
$PComS(n,1,2)$ with $n=12u^{2}-16u+6$ and $n=12u^{2}-4u+2$, respectively

\begin{theorem}
The number of families $PComS(12u^{2}-16u+6,1,2)$ in $\G_{6u^{2}-8u+3}(a)\times\G_{6u^{2}-8u+3}(b)$ with $\frac{3u^{2}-7u+4}{2}\leq a\leq\frac{9u^{2}-15u+2}{2}$ and $b=6u^{2}-11u+5-a$ is bounded 
from above by
\begin{eqnarray}
\sum_{h=3u-3}^{3u^{2}-4u}\sum_{a=\frac{3u^{2}-7u+4}{2}}^{\frac{9u^{2}-15u+2}{2}}&&\sum_{d=0}^{h}\sum_{a_{1}=0}^{a}\binom{a_{1}}{d,i_{2},...,i_{r}}\binom{a-a_{1}}{(h-d),j_{2},...,j_{s}}\\
&\times&\sum_{e=0}^{3u^{2}-4u+1-h}\sum_{b_{1}=0}^{6u^{2}-11u+5-a}\binom{b_{1}}{e,k_{2},...,k_{t}}\binom{6u^{2}-11u+5-a-b_{1}}{(3u^{2}-4u+1-h-e),l_{2},...,l_{w}}\nonumber
\end{eqnarray}
where $i_{2}+\cdots+i_{r}=a_{1}-d$, $i_{g}\geq2$, $j_{2}+\cdots+j_{s}=a-a_{1}-h+d$, $j_{g}\geq2$ and
$k_{2}+\cdots+k_{t}=b_{1}-e$, $k_{g}\geq2$, $l_{2}+\cdots+l_{w}=3u^{2}-7u+4-a-b_{1}+h+e$, 
$l_{g}\geq2$.
\end{theorem}

\begin{theorem}
The number of families $PComS(12u^{2}-4u+2,1,2)$ in $\G_{6u^{2}-2u+1}(a)\times\G_{6u^{2}-2u+1}(b)$ with $\lfloor\frac{3u^{2}-4u+3}{2}\rfloor\leq a\leq\lfloor\frac{3u^{2}+2u}{2}\rfloor$ and 
$b=6u^{2}-5u+2-a$ is bounded from above by
\begin{eqnarray}
\sum_{h=3u-2}^{3u^{2}-u+1}\sum_{a=\lfloor\frac{3u^{2}-4u+3}{2}\rfloor}^{\lfloor\frac{3u^{2}+2u}{2}\rfloor}&&\sum_{d=0}^{h}\sum_{a_{1}=0}^{a}\binom{a_{1}}{d,i_{2},...,i_{r}}\binom{a-a_{1}}{(h-d),j_{2},...,j_{s}}\\
&\times&\sum_{e=0}^{3u^{2}-2u-h}\sum_{b_{1}=0}^{6u^{2}-5u+2-a}\binom{b_{1}}{e,k_{2},...,k_{t}}\binom{6u^{2}-5u+2-a-b_{1}}{(3u^{2}-2u-h-e),l_{2},...,l_{w}}\nonumber
\end{eqnarray}
where $i_{2}+\cdots+i_{r}=a_{1}-d$, $i_{g}\geq2$, $j_{2}+\cdots+j_{s}=a-a_{1}-h+d$, $j_{g}\geq2$ and
$k_{2}+\cdots+k_{t}=b_{1}-e$, $k_{g}\geq2$, $l_{2}+\cdots+l_{w}=3u^{2}-3u+2-a-b_{1}+h+e$, 
$l_{g}\geq2$.
\end{theorem}

\section{Conclution}
In this paper was established the relationship between the Schur rings, the run structure 
and the periodic compatible binary sequences by means of using the autocorrelation function. We 
characterize the run structure of families $PComS$ and with this construction we generalize
combinatorial structures that have been studied for many years such as Hadamard matrices and 
perfect binary sequences. In this way we conclude that it is of great importance to study
Schur rings over $\Z_{2}^{n}$ whose basic sets are orbits of abelian subgroups from group 
of automorphism of $\Z_{2}^{n}$. Also, we note that the conjectures 1 to 6 are all related and, 
therefore it is necessary to research such Schur rings. On the other hand, the goal by defining
families $PComS$ is create the need of developing construction methods and applications of 
structure which are neither Hadamard matrices nor perfect binary sequences.

\section{Appendix}

All nontrivial families $PComS$, where $n$ ranges in $5,6,7,8,9$, are showed.

\subsection{$PComS$ in $\mathbb{Z}_{2}^{5}$}
\begin{enumerate}
\item $PComS(5,1,1)=\{(+----)_{C}\}$,
\item $PComS(5,2,-2)=\{(++---)_{C},(+--+-)_{C}\}$.
\end{enumerate}

\subsection{$PComS$ in $\mathbb{Z}_{2}^{6}$}
\begin{enumerate}
\item $PComS(6,1,2)=\{(+-----)_{C}\},$
\item $PComS(6,3,-2)=\{(++----)_{C},(+-+---)_{C},(++-+--)_{C}\},$
\end{enumerate}

\subsection{$PComS$ in $\mathbb{Z}_{2}^{7}$}
\begin{enumerate}
\item $PComS(7,1,3)=\{(+------)_{C}\}$,
\item $PComS(7,1,-1)=\{(++-+---)_{C}\},$
\item $PComS(7,3,1)=\{(++-----)_{C},(+-+----)_{C},(+---+--)_{C}\}$,
\item $PComS(7,3,-3)=\{(+++----)_{C},(++--+--)_{C},(+-+-+--)_{C}\}.$
\end{enumerate}

\subsection{$PComS$ in $\mathbb{Z}_{2}^{8}$}
\begin{enumerate}
\item $PComS(8,1,4)=\{(+-------)_{C}\}.$
\item $PComS(8,2,0)=\{(+-+-----)_{C},(+++--+--)_{C}\}.$
\item $PComS(8,3,0)=\{(+-+-----)_{C},(++-+----)_{C},(++--+---)_{C}\}.$
\item $PComS(8,4,0)=\{(++------)_{C},(+--+----)_{C},(+++-+---)_{C},\\
(+--++-+-)_{C}\}$
\item $PComS(8,4,-4)=\{(+--++-+-)_{C},(++---++-)_{C},(+++-+---)_{C},\\
(++----+-)_{C}\}.$
\item $PComS(8,5,4)=\{(++------)_{C},(+-+-----)_{C},(+--+----)_{C},\\
(+---+---)_{C},(++-+----)_{C}\}$
\item $PComS(8,9,-8)=\{(++++----)_{C},(++-+--+-)_{C},(+++-+---)_{C},\\
(+--++-+-)_{C},(+++---+-)_{C},(+++--+--)_{C},(++--+-+-)_{C},\\
(++---++-)_{C},(+---+---)_{C}\}$.
\end{enumerate}

\subsection{$PComS$ in $\mathbb{Z}_{2}^{9}$}
\begin{enumerate}
\item $PComS(9,1,5)=\{(+--------)_{C}\}$,
\item $PComS(9,2,-2)=\{(+++-+----)_{C},(+-+--++--)_{C}\}$
\item $PComS(9,2,-2)=\{(+++--+---)_{C},(++-+---+-)_{C}\}$
\item $PComS(9,2,-2)=\{(++---+-+-)_{C},(++----++-)_{C}\}$
\item $PComS(9,4,-4)=\{(++++-----)_{C},(++---++--)_{C},(+--+-+-+-)_{C},\\
(++-+--+--)_{C}\}$
\item $PComS(9,4,0)=\{(+++------)_{C},(+--+-----)_{C},(++---+-+-)_{C},\\
(+-+--++--)_{C}\}$
\item $PComS(9,4,8)=\{(++-------)_{C},(+-+------)_{C},(+--+-----)_{C},\\
(+---+----)_{C}\}$
\item $PComS(9,5,1)=\{(+--+-----)_{C},(++-+--+--)_{C},(+++------)_{C},\\
(++---+---)_{C},(+----+-+-)_{C}\}$
\item $PComS(9,6,6)=\{(++-------)_{C},(+-+------)_{C},(+---+----)_{C},\\
(++-+-----)_{C},(+----++--)_{C},(+--+---+-)_{C}\}$
\item $PComS(9,7,-1)=\{(+--+-----)_{C},(+-+--++--)_{C},(++-+---+-)_{C},\\
(++----++-)_{C},(+++------)_{C},(++---+---)_{C},(+----+-+-)_{C}\}$
\item $PComS(9,7,-1)=\{(++-+-----)_{C},(+----++--)_{C},(+--+---+-)_{C},\\
(++-+--+--)_{C},(+++------)_{C},(++---+---)_{C},(+----+-+-)_{C}\}$
\item $PComS(9,9,-3)=\{(+++------)_{C},(++---+---)_{C},(+----+-+-)_{C},\\
(++-+-----)_{C},(+----++--)_{C},(+--+---+-)_{C},(+-+--++--)_{C},\\
(++-+---+-)_{C},(++----++-)_{C}\}$
\item $PComS(9,10,2)=\{(+++------)_{C},(++---+---)_{C},(+----+-+-)_{C},\\
(++-+-----)_{C},(+----++--)_{C},(+--+---+-)_{C},(++----+--)_{C},\\
(+--+-+---)_{C},(++-----+-)_{C},(+--+-----)_{C}\}$
\end{enumerate}

\Addresses


\begin{thebibliography}{99}

\bibitem{} I. Schur. {\itshape Zur Theorie der einfach transitiven Permutationsgruppen},  Sitzungsber. Preuss. Akad. Wiss., Phys.-Math. Kl., 598–623, 1993.

\bibitem{} H. Wielandt. {\itshape Finite Permutation Groups}, Academic Press, New York-London, 1964.

\bibitem{} M. Klin, R. Poschel. {\itshape The konig problem, the isomorphism problem for cyclic
graphs and the method of schur rings}, Algebraic Methods in Graph Theory, 1, 2, 1978.

\bibitem{} S. L. Ma. {\itshape On association schemes, schur rings, strongly regular graphs and partial difference sets}, Ars Combin., 21:211-220, 1989.

\bibitem{} A. Heinze. {\itshape Applications of Schur rings in algebraic combinatorics:
graphs, partial difference sets and cyclotomic schemes}, PhD thesis, Universitat
Oldenburg, 2001.

\bibitem{} M. Muzychuk, M. Klin, R. Poschel. {\itshape The isomorphism problem for circulant
graphs via Schur ring theory}, Dis. Math. The. Com. Sci. \textbf{56}, 241-264, 2001.

\bibitem{} Ka Hin Leung, Shin Hing Man. {\itshape On schur rings over cyclic groups ii}, Journal of
Algebra, \textbf{183}:273-285, 1996.

\bibitem{} Ka Hin Leung, Shin Hing Man. {\itshape On schur rings over cyclic groups}, Israel Journal of Mathematics, \textbf{106}:251-267, 1998.

\bibitem{} R. Orozco, {\itshape Schur Ring over Group $\Z_{2}^{n}$, Circulant $S-$Sets Invariant by Decimation and Hadamard Matrices}, arXiv:1802.05788, 2018.

\bibitem{} R. Orozco, {\itshape An Approximation to Proof of the Circulant Hadamard Conjecture}, arXiv:1804.05007, 2018.

\bibitem{} K. Cai, {\itshape Autocorrelation-Run Formula for Binary Sequences}, arXiv:0909.4592, 2009.

\bibitem{} R.C. Tistworth, {\itshape Equivalence Classes of Periodic Sequences}, Illinois J. Math. \textbf{8}, 266-270, 1964.

\bibitem{} L. Bomer, M. Antwiler {\itshape Periodic Complementary Binary Sequences}, IEEE Trans. Inform. Theory, \textbf{36}, 1487-1494, 1990.

\bibitem{} KT. Arasu, Q. Xiang, {\itshape On the Existence of Periodic Complementary Binary Sequences}, Desing, Codes and Cryptography, \textbf{2}, 257-262, 1992.

\bibitem{} I.S. Kotsireas, C. Koukouvinos, J. Seberry, {\itshape Hadamard ideals and Hadamard matrices with circulant core}, J. Combin. Math. Comput., \textbf{57}, 47-63, 2006.

\bibitem{} I.S. Kotsireas, C. Koukouvinos, J. Seberry, {\itshape Hadamard ideals and Hadamard matrices with two circulant cores}, European J. Combin., \textbf{27}, 658-668, 2006.

\bibitem{} R. Fletcher, M. Gysin, J. Seberry, {\itshape Application of the discrete Fourier transform to the search for generalized Legendre pairs and Hadamard matrices}, Austra. J. Combin., \textbf{23}, 75-86, 2001.

\bibitem{} J. Williamson, {\itshape Hadamard's determinat theorem and the sum of four squares}, Duke Mathematical Journal, \textbf{11}, 65-81, 1994.

\bibitem{} J.M. Goethals, J.J. Seidel, {\itshape A Skew Hadamard matrix of order 36}, J. Austral. Math. Soc, \textbf{11}, 343-344, 1970.

\bibitem{} W.H. Holzmann, H. Kharaghani, B. Tayfeh-Reaie, {\itshape Williamson matrices up to order 59}, Des. Codes Cryptogr, \textbf{46}, 343-352, 2008.

\bibitem{} D.Z. Dokovic, I.S. Kotsireas, {\itshape Goethals-Seidel difference families with symmetric or skew base blocks}, arXiv: 1802.00556v1, 2018.

\bibitem{} W. de Launey, {\itshape On the asymptotic existence of partial complex Hadamard matrices and related combinatorial objects}, Discrete Appl. Math., \textbf{102}, 37-45, 2000.

\bibitem{} W. de Launey, D. Levin, {\itshape A Fourier-analytic approach to counting partial Hadamard matrices},  Cryptogr. Commun. 2, \textbf{2}, 307-334, 2010.

\bibitem{} Y. Lin, F. Phoa, M. Kao {\itshape Circulant partial Hadamard matrices: construction via general difference sets and its application to fMRI experiments}, Statist. Sinica. 27, \textbf{4}, 1715-1724, 2017.

\bibitem{} D. Jungnickel, A. Pott, {\itshape Perfect and almost perfect sequences}, Discrete Appl. Math., \textbf{95}, 331-359, 1999.


\end{thebibliography}
\end{document}